\documentclass[]{article}

\usepackage{amsmath,mathtools,amsthm,amsfonts}
\usepackage{a4wide}
\usepackage{amssymb, authblk}
\usepackage[mathscr]{eucal}
\usepackage{bm}
\usepackage{bbm}
\usepackage[shortlabels]{enumitem}
\usepackage{hyperref}
\usepackage[usenames,dvipsnames]{xcolor}

\usepackage{tikz}
\usepackage{graphicx, subcaption}
\usepackage{todonotes}

\newcommand{\prob}{\mathbb{P}}
\newcommand{\Prob}[1]{\prob\left(#1\right)}

\newcommand{\expec}{\mathbb{E}}
\newcommand{\Exp}[1]{\expec\left[#1\right]}

\newcommand{\Var}[1]{\textup{Var}\left(#1\right)}

\newcommand{\ind}[1]{\mathbbm{1}_{\left\{#1\right\}}}

\newcommand{\cs}[1]{\textcolor{blue}{\small({\sf Clara:} {{#1}})}}
\newcommand{\bz}[1]{\textcolor{red}{\small({\sf Bert:} {{#1}})}}

\numberwithin{equation}{section}

\allowdisplaybreaks

\newtheorem{theorem}{Theorem}[section]

\newtheorem{lemma}[theorem]{Lemma}
\newtheorem{proposition}[theorem]{Proposition}

\begin{document}

\title{Large deviations for triangles in scale-free random graphs}
\author{Clara Stegehuis and Bert Zwart }
\date{\today}

\maketitle

\begin{abstract}
    We provide large deviations estimates for the upper tail of the number of triangles in scale-free inhomogeneous random graphs where node degrees have tails which are power laws with index $-\alpha, \alpha \in (1,2)$. We show that upper tail probabilities for triangles undergo a phase transition. For $\alpha<4/3$, the upper tail is caused by many vertices of degree of order $n$, and this probability is semi-exponential. Additional triangles consist of two hubs. For $\alpha>4/3$ on the other hand, the upper tail is caused by one hub of a specific degree, and this probability decays polynomially in $n$, leading to additional triangles with one hub. In the intermediate case $\alpha=4/3$, we show polynomial decay of the tail probability caused by multiple but finitely many hubs. In this case, the additional triangles contain either a single hub or two hubs. Our proofs are partly based on various concentration inequalities. In particular, we tailor concentration bounds for empirical processes to make them well-suited for analyzing heavy-tailed phenomena in nonlinear settings. 
\end{abstract}

\section{Introduction and main results}
\label{sec:intro}
Many real-world networks were found to have degree distributions that can be approximated by a power-law distribution, where the fraction of vertices of degree $k$ scales as a power law with infinite variance~\cite{vazquez2002}. Therefore, random graph models that serve as benchmarks for these real-world networks often focus on networks with power-law degree distributions. These random graphs are constructed to have similar degree distributions as real-world networks, but the model does not prescribe other graph properties. The behavior of network properties of random graph models with a prescribed degree sequence has therefore been an object of intensive study~\cite{dhara2019,friedrich2015a,heydari2017,janssen2019,hofstad2017,yin2019a}.

In this paper, we focus on the property of triangle counts. Triangle counts measure the tendency of two neighbors of a vertex to be connected as well, allowing to analyze the network's clustering properties. While many real-world networks were found to be highly clustered, many random graph models are locally tree-like, and therefore only contain few triangles in the large-network limit. In power-law random graphs however, the random graphs may still possess a polynomial number of triangles, and the average clustering coefficient vanishes extremely slowly in the network size~\cite{hofstad2017b}. 
Motivated by this slow decay of the average clustering coefficient, we focus on the question: How
unlikely is it that a power-law random graph contains a large number of triangles? 

The tail probability for triangle counts in Erd\H os-R\'enyi random graphs has been studied extensively since~\cite{janson2004,janson2004a,kim2004}, and a matching upper and lower bound were finally provided in~\cite{chatterjee2012missing,demarco2011} when the random graphs are not too sparse, that is, when the average degree of each vertex tends to infinity in the network size; see also \cite{Augeri2020}. For sparser Erd\H os-R\'enyi graphs with finite average vertex degrees,~\cite{chakrabarty2021} showed that the probability of observing many triangles is extremely small, and that a localized almost clique structure drives the unlikely event that a large number of triangles is present. 

Significantly fewer results exist on large deviations for random graphs with heavy-tailed degrees. Most existing work assumes light-tailed degree distributions, or a finite second moment~\cite{andreis2021,chakrabarty2021,dommers2018}, as this allows to write the connection probability in a product form of the weights of the two involved nodes, while an infinite second moment makes this impossible. Under an infinite second moment, the connection probability depends on both vertex weights in a way that cannot be split into their individual contributions, creating so-called degree-degree correlations~\cite{stegehuis2017b,yao2017}.
Other work focuses on a regime where the average degree grows~\cite{oliveira2019}, which is not applicable when the tail of the degree distribution behaves like a power law with index $-\alpha, \alpha \in (1,2)$. Results on large-deviations analysis for power-law random graphs are so far restricted to the Pagerank functional \cite{Litvak2017, Mariana2021}, 
and edge counts \cite{kerriou2022, stegehuis2022scale}.

In this paper, we derive tail asymptotics for the probability that the number of triangles is larger than average for the sparse, power-law case with $\alpha \in(1,2)$. Interestingly, we show that there is a phase transition in the degree exponent $\alpha$. When $\alpha < 4/3$, the probability that the number of triangles is larger than expected decays semi-exponential in the network size $n$. For $\alpha\geq 4/3$ on the other hand, this probability decays polynomially in $n$. Furthermore, in contrast to the non-power-law case, deviations of the triangle counts are caused by the presence of one or more hubs of specific degree.

\paragraph{Notation.} Let $f(n) \sim g(n)$ denote $f(n)/g(n)\rightarrow 1$. We say that $U$ is a regularly varying function with index $\rho$ when for every $\lambda>0$, $\lim_{x\to\infty}U(\lambda x)/U(x)=\lambda^\rho$. Furthermore, we say that a sequence of events $\{\mathcal{E}_n\}_{n\geq 1}$ happens with high probability when $\lim_{n\to\infty}\Prob{\mathcal{E}_n}=1.$ We denote the indicator function by $I(\cdot)$.

\subsection{Model description}
To give a precise description of our main results, we now provide a model description. 
We consider the rank-1 inhomogeneous random graph (or hidden variable model). This model constructs simple graphs with soft constraints on the degree sequence~\cite{boguna2003,chung2002}. The graph consists of $n$ vertices with non-negative weights $W_i, i=1,...,n$. These weights are an i.i.d.\ sample from the continuous heavy-tailed distribution $F(x),$ 
	\begin{equation}
		\label{D-tail}
		\bar{F}(x) := \Prob{W>x}= x^{-\alpha}L(x), \quad x\geq 0
	\end{equation}
 for some slowly varying function $L(x)$, $\alpha \in (1,2)$. 

We denote  $\mu = \Exp{W_i}$. Then, every pair of vertices with weights $(h,h')$ is independently connected with probability $p(h,h')$. In this paper, we take
	\begin{equation}\label{eq:phh}
		p(h,h')=\min\left(\frac{hh'}{\mu n},1\right),
	\end{equation}
	which is the Chung-Lu version of the rank-1 inhomogeneous random graph~\cite{chung2002}. This connection probability ensures that the degree of a vertex with weight $h$ will be close to $h$~\cite{boguna2003}. 
	
	We are interested in the number of triangles $\triangle_n$ contained in a sample of the rank-1 homogeneous random graph. Denote
 \begin{equation}\label{eq:fn}
     f_n(u,v,w) = \min \Big\{ \frac {uv}{\mu n},1 \Big\}   \min \Big\{ \frac {vw}{\mu n},1 \Big\}   \min \Big\{ \frac {uw}{\mu n},1 \Big\}.
 \end{equation}
	The next result, proven in Appendix \ref{app:sec14appendix}, describes the growth rate of $m_n=\Exp{\triangle_n}$, extending previous work ~\cite{stegehuis2019b, hofstad2017d} on the pure power law case.

	\begin{lemma}
\label{lemma-mean} Let $\alpha \in (1,2)$. 
$m_n = (1+o(1))H n^3 (\bar F(\sqrt{n}))^3$, with 
\begin{equation}
    H=\frac{\alpha^3}{6} \int_{u=0}^\infty\int_{v=0}^\infty\int_{w=0}^\infty f_1(u,v,w)  u^{-\alpha-1} v^{-\alpha-1} w^{-\alpha-1} dudvdw.
\end{equation}
In particular, $m_n$ is regularly varying with index $3- \frac 32 \alpha$.
\end{lemma}


%

	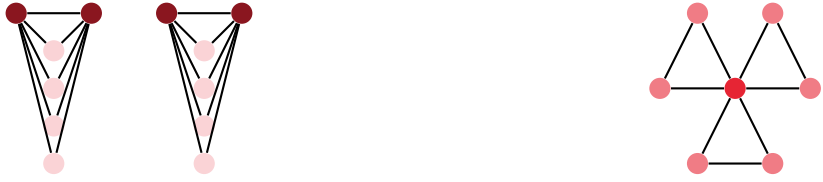
\begin{figure}[tb]
	\definecolor{mycolor1}{RGB}{230,37,52}%
	\tikzstyle{every node}=[circle,fill=black!25,minimum size=8pt,inner sep=0pt]
	\tikzstyle{S1}=[fill=mycolor1!60]
	\tikzstyle{S2}=[fill=mycolor1!60!black]
	\tikzstyle{S3}=[fill=mycolor1]
	\tikzstyle{n1}=[fill=mycolor1!20]
	\begin{subfigure}[t]{0.46\linewidth}
		\centering
		\begin{tikzpicture}
		\tikzstyle{edge} = [draw,thick,-]
		\node[S2] (a) at (0,0) {};
		\node[S2] (b) at (1,0) {};
		\node[n1] (c) at (0.5,-0.5) {};
		\node[n1] (d) at (0.5,-1) {};
		\node[n1] (e) at (0.5,-1.5) {};
		\node[n1] (f) at (0.5,-2) {};
		\draw[edge] (a)--(b);
		\draw[edge] (c)--(b);
		\draw[edge] (d)--(b);
		\draw[edge] (a)--(c);
		\draw[edge] (a)--(d);
		\draw[edge] (e)--(b);
		\draw[edge] (f)--(b);
		\draw[edge] (e)--(a);
		\draw[edge] (f)--(a);
		
		\node[S2] (a1) at (2,0) {};
		\node[S2] (b1) at (3,0) {};
		\node[n1] (c1) at (2.5,-0.5) {};
		\node[n1] (d1) at (2.5,-1) {};
		\node[n1] (e1) at (2.5,-1.5) {};
		\node[n1] (f1) at (2.5,-2) {};
		\draw[edge] (a1)--(b1);
		\draw[edge] (c1)--(b1);
		\draw[edge] (d1)--(b1);
		\draw[edge] (a1)--(c1);
		\draw[edge] (a1)--(d1);
		\draw[edge] (e1)--(b1);
		\draw[edge] (f1)--(b1);
		\draw[edge] (e1)--(a1);
		\draw[edge] (f1)--(a1);
		\end{tikzpicture}	
		\caption{$\alpha<4/3$ or more than $n$ triangles: exponential deviations, caused by many hubs of weight of $O(n)$ that form many triangles with one node of constant weight.}
		\label{fig:expdev}
	\end{subfigure}
	\hfill
	\begin{subfigure}[t]{0.46\linewidth}
		\centering
		\begin{tikzpicture}
		\tikzstyle{edge} = [draw,thick,-]
		\node[S3] (a) at (0,0) {};
		\node[S1] (b) at (1,0) {};
		\node[S1] (c) at (0.5,1) {};
		\node[S1] (d) at (-0.5,1) {};
		\node[S1] (e) at (-1,0) {};
		\node[S1] (f) at (-0.5,-1) {};
		\node[S1] (g) at (0.5,-1) {};
		\draw[edge] (a)--(b);
		\draw[edge] (c)--(b);
		\draw[edge] (d)--(a);
		\draw[edge] (a)--(c);
		\draw[edge] (a)--(e);
		\draw[edge] (a)--(f);
		\draw[edge] (a)--(g);
		\draw[edge] (d)--(e);
		\draw[edge] (f)--(g);
		\end{tikzpicture}	
		\caption{$\alpha>4/3$: polynomial deviations, caused by one hub of weight of $O(n^{(\alpha+2\theta)/(4(\alpha-1))})$ 
		that forms triangles with vertices of lower degrees. $\theta=0$ in Theorem \ref{thm-triangle-singlebigjump}.} 
		\label{fig:poldev}
	\end{subfigure}
	\caption{Illustration of the events that cause polynomial and exponential deviations.}
	\end{figure}



\subsection{Main results and discussion}\label{sec:results}

%

We are interested in the event that $\triangle_n$ deviates from its mean $m_n$ by a factor $a>0$. It turns out that there is a qualitative difference determined by the question whether 
$m_n$ is increasing to $\infty$ faster than $n$ or not, which, due to Lemma \ref{lemma-mean}, is determined by whether the parameter $\alpha$ is bigger than or smaller than $4/3$. 

\paragraph*{The case $\alpha>4/3$: single hub}
As we will see, an unusually large number of triangles will be caused by one hub. To determine how large a hub needs to be, we define for $a>0$,

\begin{equation}
\label{definition-can}
c_a(n) := \inf \Big\{ c: n^2 \frac 12 \int_0^\infty \int_0^\infty f_n(x,y,c) dF(x)dF(y) \geq m_n a \Big\}.
\end{equation}
Intuitively, $c_a(n)$ is the smallest size of a hub needed to create $a m_n$ additional triangles. The next lemma estimates its order of magnitude. Its proof is given in Appendix \ref{app:sec14appendix}. 
\begin{lemma}
\label{lemma-can}
$c_a(n)$ is regularly varying of index $\beta = \frac 14 \frac{\alpha}{\alpha-1}$. 
In particular, there exists a slowly varying function $L^*$ such that 
\begin{equation}
    c_a(n) \sim L^*(n) n^\beta a^{\frac 12\frac 1{\alpha-1}}
\end{equation}
for every $a>0$.
\end{lemma}
The index $\beta$ equals $1$ at $\alpha=4/3$ and decreases in $\alpha$. $\beta$ equals $1/2$ when $\alpha=2$. One can show using extreme-value theory that the typical value of the largest weight in the random graph is regularly varying with index $1/\alpha$. For all $\alpha \in (1,2)$ we have $\beta > 1/\alpha$. A hub of size $c_a(n)$ is therefore a rare event, and our first main theorem confirms that it is the most likely rare event leading to $(1+a) m_n$ triangles when $\alpha>4/3$. 


\begin{theorem}  
\label{thm-triangle-singlebigjump}
Let $\alpha > 4/3$ and $a>0$. As $n\rightarrow\infty$, 
\begin{equation} 
\Prob{\triangle_n > (1+a) m_n} = (1+o(1)) n \Prob{ W > c_a(n)}. 
\end{equation} 
\end{theorem} 


By applying a concentration result of  Chatterjee~\cite{chatterjee2012missing}, we show in Section \ref{sec:prooftaularge} that, to establish the behavior of  $\Prob{\triangle_n > (1+a)m_n}$, it suffices to establish the behavior of $\Prob{G_n > (1+a)m_n}$, with 
\begin{equation}
\label{eq:def:gn}
    G_n = \Exp{\triangle_n \mid W_1,...,W_n}.
\end{equation}
To analyze $\Prob{G_n > (1+a)m_n}$, we extend ideas from heavy-tailed large deviations theory (see e.g.\ 
\cite{Collamore2018,  fosskorshunov2012,  Mikosch2013,  rheeblanchetzwart2016,  zwart2004})
 and formalize the intuition that a single big hub is needed. A major step
is to show that the event $\{G_n > (1+a) m_n\}$ is much more unlikely when all nodes have weight smaller than $\varepsilon c_a(n)$ for some suitable $\varepsilon>0$. 
However, existing heavy-tailed large deviation tools focus on essentially linear processes and are less suitable to apply to functionals of random graphs which are essentially
nonlinear as in (\ref{eq:def:gn}).


For this reason we develop a different approach: we write the number of triangles as a functional of the empirical distributions of the weights as in (\ref{definition-gn}),
and we derive a novel concentration result (Proposition \ref{prop-empiricalconcentration}), building on a classical concentration result for weighted empirical processes \cite{Wellner1978}.
For a more precise statement we refer to (\ref{eq-deffnstar}). This approach seems promising for
other nonlinear functionals of heavy-tailed random variables, like $U$-statistics, or other observables of random graphs.
Examples of recent work on nonlinear large deviations in a light-tailed setting are \cite{Augeri2020,   ChatterjeeDembo2016}.


\paragraph*{The case $\alpha<4/3$: many hubs} For $\alpha<4/3$, the probability of a larger than average number of triangles is semi-exponential instead:
\begin{theorem}  
\label{thm:smalldevs}
Suppose that $\Prob{W\geq 1}=1$ and $\Prob{W>x} \sim C x^{-\alpha}$ for $x\geq 1$, $\alpha\in (1,4/3)$. For any fixed $a>0$,
\begin{align}
  \lim_{n\to\infty}\frac {\log  \Prob{\triangle_n>n^{(3-\alpha 3/2)}(C^3H+a)}}{n^{1-\alpha 3/4}\log(n)}
   =-\sqrt{2 a}\frac{\alpha}{4}.
\end{align}
\end{theorem}


In this setting, a large number of triangles is caused by a sublinear (but polynomially growing in $n$) number of vertices of weight $\mu n$. If $\Prob{W\geq 1}=1$, a vertex of weight $\mu n$ connects to all other vertices with probability one. Thus, in pairs of two, these hubs form triangles with all other vertices, see Figure~\ref{fig:expdev}. 

The square root of $a$ can intuitively be explained by the fact that $B$ vertices of weight $\mu n$ create $nB(B-1)/2\approx nB^2/2$ triangles. Indeed, each of the $B^2/2$ pairs of vertices of weight at least $\mu n$ creates triangles with each of the $n$ other vertices. Thus, with $B=\sqrt{2 a} n^{1-\alpha 3/4}$ create at least $an^{(3-\alpha 3/2)}$ triangles.

To prove an asymptotic upper bound, we will split all triples of nodes into three sets based on the triangle weights
and then construct an upper bound on the triangle counts with edges in these weight groups one by one. Here we will use similar concentration bounds as developed in the proof of Theorem \ref{thm-triangle-singlebigjump} to get rid of the random weights, as long as these weights are sufficiently small. For triangles containing larger weights, we will use properties of the function $f_n$ instead to deal with the multiple sources of randomness.

The phenomenon of having a number of big hubs that is growing with $n$ is non-standard in the context of heavy tails; a related example appears in an exit problem for the sample average of a random walk, where the number of big values required to avoid escaping a convex set for $n$ time
units is logarithmically increasing with $n$ \cite{ayan}.

We believe that it is possible to extend Theorem~\ref{thm:smalldevs} to allow for non-trivial slowly varying functions.
In this setting, the denominator of the scaling will likely also include a term with $L(n)$, the slowly varying function evaluated at $n$, due to the fact that the main contribution is from vertices of weight $n$. However, in the current proof we distinguish different types of triangles at different scales of the weights. Proving such a statement with slowly varying functions then entails showing that the contribution of the slowly varying functions at other scales that will appear in the probabilities of these non-dominating triangles are small compared to the contribution of $L(n)$, which becomes rather technical, especially when $L(n)$ oscillates.

\paragraph*{Behavior at the boundary $\alpha=4/3$: multiple hubs} The above two theorems show a stark contrast in the way additional triangles are generated: if $\alpha>4/3$ they consist of two regular nodes and one hub of the order of~\eqref{definition-can}, and if $\alpha<4/3$, they consist one regular node, and two hubs of the order of~\eqref{definition-can}. At the boundary $\alpha=4/3$ both may occur.
%
When $\alpha=4/3$,  $c_a(n)$ defined in (\ref{definition-can}) is regularly varying of index $1$, as is $m_n$. 
To avoid technical complications with slowly varying functions that can arise on the boundary (for example, $m_n/n$ could be oscillating between 0 and $\infty$), we assume that $\Prob{W>n} = (1+o(1))C n^{-\alpha}$, in which case Lemma \ref{lemma-mean} $m_n \sim C^3Hn$.

Depending on the value of $a$, a single big value of the weights $W_i$ may not be enough to create $(1+a)m_n$ triangles, for which we now provide some intuition. If there are $l$ hubs with a weight of infinite 
size and $n$ regular nodes, each hub forms a triangle
with any of the $\mu n$ edges,
leading to $nl\mu/2$ additional triangles consisting of a single hub and two regular nodes. In addition, each of the $l(l-1)/2$ pairs of hubs form $n$ triangles with the regular nodes. Therefore, if 
we wish to exceed the number of triangles with a factor $a m_n$ we need $k(a)$  hubs where $k(a)$ is defined as
\begin{equation}
\label{eq-def-k(a)}
    k(a) := \inf \{l: l \mu/2 + l(l-1)/2 > a C^3 H \}.
\end{equation}
To derive a precise result, we need to take into account that hubs have weight of $O(n)$ rather than $\infty$. To this end, we define
\begin{align}
    K_l(z_1,...,z_l) & = \frac 1{2\mu} \sum_{i=1}^l (\frac{z_i}{\mu}\Exp{W^2I(W\leq \mu/z_i)}+\Exp{WI(W>\mu/z_i}))^2 \nonumber\\
    & \quad + 
    \sum_{i=1, j>i}^l \Exp{\min \{\frac{z_i}{\mu}W,1\}, \min \{\frac{z_j}{\mu}W,1\}},
\end{align}
where $I$ denotes the indicator function. 
As shown in Lemma \ref{lem-lb-additionaltriangles} below, (see also Proposition \ref{prop-hublimit}), $nK_l(z_1,...,z_l)$ can be interpreted as the expected number of additional triangles (up to a term of $o(n)$) caused by $l$ hubs of size $z_i n, i=1,...,l$. 

We can now formulate our main theorem for $\alpha=4/3$. Let, for $b>0$, $X_i^b, i\geq 1$, be an i.i.d.\ sequence such that $\prob(X_i^b>x) = (x/b)^{-\alpha}, x \geq b$.
Define $\eta(a)$ as 
\begin{equation}
\label{def:eta(a)}
    \eta(a) = \inf \{ \eta: K_{k(a)}(\eta, \infty, ...,\infty) \geq  C^3Ha\}.
\end{equation}
Note that 
$\eta(a)>0$ if $K_{k(a)-1}(\infty, ...,\infty)  = (k(a)-1) \mu + (k(a)-1)(k(a)-2)/2 <  a C^3 H$. 
\begin{theorem}
\label{thm-triangleboundary}
    Suppose that $P(W>x)\sim C x^{-4/3}$,
and suppose that  $(k(a)-1) \mu + (k(a)-1)(k(a)-2)/2 <  a C^3 H$.
     Then 
     \begin{equation}
   \Prob{\triangle_n > (1+a) m_n} \sim  
   \prob( K_{k(a)}(X_1^{\eta(a)}, \ldots, X_{k(a)}^{\eta(a)}) \geq C^3 H a) (n\Prob{W> \eta(a) n})^{k(a)}.
\end{equation}

\end{theorem}

We prove Theorem \ref{thm-triangleboundary} along similar lines as Theorem \ref{thm-triangle-singlebigjump}, namely by first showing the analogous result for $G_n$. The technical condition $(k(a)-1) \mu + (k(a)-1)(k(a)-2)/2 <  a C^3 H$ is needed to be able to pin down the number of hubs that is required for $G_n$ and $\triangle_n$ to be large.

\paragraph*{Larger deviations}
We now show that  larger deviations of an order of magnitude $n^\theta$ from the mean again induce a phase transition. 
First, the probability that a factor of $n^\theta$ more triangles than average are present decays regularly varying in $n$ up until $\theta=(3/2)\alpha-2$:
\begin{theorem}\label{thm:thetasmall}
Let $\alpha>4/3$ and $\theta \in (0, \frac 32 \alpha-2)$. Then $c_{n^\theta}(n) \sim L^*(n) n^{\beta+\frac\theta 2 \frac{1}{\alpha-1}}=o(n)$ 
for some slowly varying $L^*$
and 
\begin{equation}
    \Prob{\triangle_n> m_n (1+n^\theta)} \sim n\Prob{W> c_{n^\theta}(n)}.
\end{equation}

In particular, $    \Prob{\triangle_n> m_n (1+n^\theta)}$ is regularly varying with exponent $1-\alpha \beta - \alpha \frac\theta 2 \frac{1}{\alpha-1}$.
\end{theorem}

The proof of Theorem \ref{thm:thetasmall} follows the same steps as the proof of Theorem \ref{thm-triangle-singlebigjump}, but is at several points
slightly more technical. For readability we provide these additional technical details separately in Appendix \ref{app:proof:thm:thetasmall}.

\begin{figure}[tbp]
    \centering
    \includegraphics[width=0.5\linewidth]{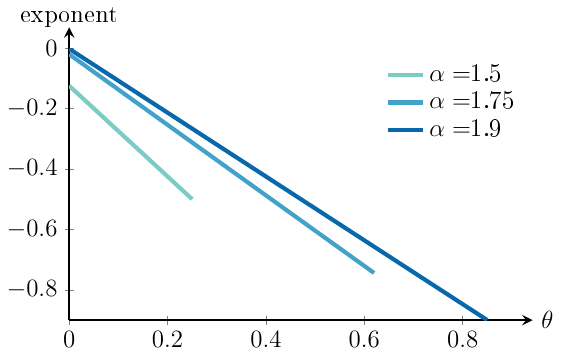}
    \caption{The exponent of Theorem~\ref{thm:thetasmall} plotted against $\theta$ for several values of $\alpha$.} 
    \label{fig:atau}
\end{figure}

Note that Theorem~\ref{thm:thetasmall} applies up to $\theta=(3/2)\alpha-2$. Thus, the higher $\alpha$, the larger the factor of deviations that can still be computed with this theorem, as also shown in Figure~\ref{fig:atau}. Furthermore, Figure~\ref{fig:atau} illustrates that for larger values of $\alpha$, a deviation of $n^\theta$ is more likely than for smaller values of $\alpha$. That is, more degree inhomogeneity makes deviations of the triangle counts more unlikely. At first sight, this may be in contrast with the intuition that degree inhomogeneity makes it more likely for extreme values of the weight sequence to appear, and therefore could make a deviation of the triangle counts more likely. However, the average number of triangles is also larger for low values of $\alpha$, so that a lower number of triangles is required for higher values of $\alpha$ to get the same deviating factor than for lower values of $\alpha$.

Combining the upper bound $(3/2)\alpha-2$ for $\theta$ with the fact that $\Exp{\triangle_n}$ is regularly varying with index $n^{3-3\alpha/2}$ shows that the theorem applies until deviations of order $n$, as $n^{3-3\alpha/2}n^{(3/2)\alpha-2}=n$. Intuitively, this is because the inhomogeneous random graph has on average $\mu n$ edges. A single high-degree vertex can therefore only create $\mu n$ triangles. However, the maximum possible number of triangles in a graph on $n$ vertices scales as $n^3$. 

We complement Theorem \ref{thm:thetasmall} by studying the cases where $\alpha<4/3$ or $\alpha>4/3$ and deviations of more than a factor of $n^{(3/2)\alpha-2}$ from average: 
\begin{theorem}\label{thm:triangldptausmall}
Suppose that $\Prob{W\geq 1}=1$ and $\Prob{W>x} \sim C x^{-\alpha}$ for $x\geq 1$. For $\gamma\in(\max(1,3-3\alpha/2),3)$ and $\alpha>1$ and $a>0$,
\begin{equation}
    \lim_{n\to\infty}\frac {\log  \Prob{\triangle_n>an^\gamma}}{n^{(\gamma-1)/2}\log(n)} =\sqrt{2a}\Big(\frac{3-\gamma}{2}-\alpha\Big)<0.
\end{equation}
\end{theorem}
The proof of this theorem is similar to the proof of Theorem~\ref{thm:smalldevs}, and can be found in Appendix~\ref{sec:prooftausmallgamma}.
The logarithmic asymptotics of Theorems~\ref{thm:thetasmall} and~\ref{thm:triangldptausmall} match at the boundary where $\gamma=1 $ or $\theta=(3/2)\alpha-2$. We expect that a similar theorem as Theorem~\ref{thm-triangleboundary}  holds in this case. 
There is a sharp phase transition for $\alpha>4/3$ between deviations up to a factor of $n^{(3/2)\alpha-2}$ (Theorem~\ref{thm:thetasmall}) and larger deviations (Theorem~\ref{thm:triangldptausmall}). Similarly to the smaller deviations, this phase transition happens when the single hub scaling equals $n$. Indeed, for Theorem~\ref{thm:thetasmall}, the driving event is one hub of magnitude $n^{(\alpha+2\theta)/(4(\alpha-1))}$. For $\theta=(3/2)\alpha-2$, this means that one hub of order $n$
is necessary. A hub weight of $\mu n$ already makes all connection probabilities equal to one, so that increasing the hub weight further will not increase the number of triangles. Thus, to create even more triangles, a larger number of hubs is necessary, explaining the phase transition between Theorems~\ref{thm:thetasmall} and Theorem~\ref{thm:triangldptausmall}.

Interestingly, in the regime of Theorem~\ref{thm:triangldptausmall}, a lower value of $\alpha$ makes the probability of $n^\theta$ more triangles than expected more likely than a higher value of $\alpha$, contrary to the regime of Theorem~\ref{thm:thetasmall}. 

\subsection{Organization of the paper} 

In Section \ref{sec:gn:alphalarge}, we analyze the behavior of $G_n$ in the case $\alpha>4/3$, after having first developed a concentration bound (Proposition \ref{prop-empiricalconcentration}).  The tail behavior of $G_n$ for the boundary case $\alpha=4/3$ is analyzed in Section \ref{sec:gn:alphaboundary}. Section \ref{sec:prooftaularge} completes the proofs of Theorem \ref{thm-triangle-singlebigjump} and \ref{thm-triangleboundary}.
The case with many hubs, in particular Theorem \ref{thm:smalldevs} is proven in Section \ref{sec:prooftausmall}. 

We collect several proofs with more standard and/or repetitive arguments in the appendices.
The proof of the lemmas presented so far, as well as proofs of various auxiliary results in the subsequent three sections are given in Appendix \ref{app:sec14appendix}. Auxiliary results for Section \ref{sec:prooftausmall} are proven in Appendix \ref{app:manyhubs}. 
Finally, Appendix \ref{sec:prooftausmallgamma} and \ref{app:proof:thm:thetasmall} contain the necessary additional details which are needed to complete the proofs of Theorem \ref{thm:triangldptausmall} and Theorem \ref{thm:thetasmall}.

\section{Nonlinear heavy-tailed large deviations}
\label{sec:gn:alphalarge}

 Let
\begin{equation}
F_n(x) = \frac{1}{n} \sum_{i=1}^n I(W_i \leq x) , \hspace{1cm} x\geq 0.
\end{equation}
be the empirical distribution function associated with the weights $W_1,\ldots, W_n$ and observe that
\begin{equation}
\label{definition-gn}
    G_n = n^3 \int_0^\infty  \int_x^\infty  \int_y^\infty f_n(x,y,z) dF_n(z) dF_n(y) dF_n(x),
\end{equation}
with $f_n$ as in~\eqref{eq:fn}.
As a convention, integration regions are always of the form $(x,\infty)$, $(y,\infty)$, etc.\ to avoid double counting. 
This section proves the following theorem. 
\begin{theorem}
    \label{thm-gn}
    If $\alpha \in (4/3,2)$, then
    \begin{equation}
        \Prob{G_n > m_n (1+a) } = (1+o(1)) n \Prob{ W > c_a(n)}. 
    \end{equation}
\end{theorem}

This theorem serves as a major stepping stone towards the proof of Theorem \ref{thm-triangle-singlebigjump}, which will be completed in Section \ref{sec:prooftaularge}. 
The proof of Theorem \ref{thm-gn} consists of the following steps.
\begin{enumerate}
    \item Building on concentration results for weighted empirical processes, dating back to \cite{Wellner1978}, we construct an event of high probability on which we can bound $F_n$ with a suitable function $F_n^*$, which is essentially a mixture of $F$, and two large (in $n$) atoms. 
    \item To effectively use $F_n^*$, we develop some estimates for the expected number of triangles generated by one or two large hubs.
    \item We combine both previous steps in constructing a sharp bound for $G_n$ which holds with high probability when there is no hub of size bigger than $\varepsilon c_a(n)$, for some $\varepsilon > 0$. 
    \item Using this sharp upper bound for $G_n$, we complete the proof of Theorem \ref{thm-gn}.
\end{enumerate}
These four steps are worked out in the next four subsections.

\subsection{Concentration of weighted empirical distribution functions} 
In this subsection, we construct a convenient upper bound for $\bar F_n(x)=1-F_n(x)$. 
Set uniform random variables $U_i = F(W_i), i\geq 1$.
Set $F_n^U$ as the empirical distribution function of the $U_i$ variables and $\bar{F}_n$ as their complementary comulative distribution function. Define 
\begin{equation}\label{eq:h}
h(x) = x (\log x-1) + 1.
\end{equation}
Then, Lemma 1 of \cite{Wellner1978} states, for $\lambda \geq 1, y\in (0,1]$,
\begin{equation}
\label{wellner1978}
    \Prob{ \sup_{t\in [y,1]} |F_n^U(t)/t| \geq \lambda } \leq  e^{-ny h(\lambda)}.
\end{equation}
Now, let $$a(n) = \bar F^{-1} (1/n) n^{-\delta},$$ for some $\delta>0.$
By taking $y=\bar F(a(n))$ and $\lambda= A+1$ in (\ref{wellner1978}) we get
\begin{equation}
\label{cbtail1}
\Prob{\sup_{x< a(n)} |\frac{\bar F_n(x)}{\bar F(x)}-1| > A} \leq e^{- n \bar F(a(n)) h(A+1)}.
\end{equation}
Let $\bar{F}(x)=1-F(x)$ and $$b(n) = \bar F^{-1} (1/n) n^{\delta}.$$
For a fixed $\delta \in (0, 1/\alpha-1)$, $c>0$, and $A>0$ define the event $E_n(A, c, \delta)$ by
\begin{equation}\label{eq:En}
\Big \{ \sup_{x< a(n)} \frac{\bar F_n(x)}{\bar F(x)} \leq  1+A,  \sup_{x\in  [a(n), b(n)]} \bar F_n(x)\leq  (1+A) \bar F(a(n)),  \sup_{x\in  [b(n), \infty)} \bar F_n(x)\leq  c/n \Big \}.
\end{equation} 
Observe that (\ref{cbtail1}) implies 
\begin{equation}
\label{cbtail2}
\Prob{ \sup_{x\in  [a(n), b(n)]} \bar F_n(x)>    (1+A) \bar F(a(n))} 
\leq 1-\Prob{E_n(A,c,\delta)}
\end{equation}
as $\bar F_n(x)$ is non-increasing.
Finally, we investigate the range $[b(n),\infty)$. Let $c>0$ and note that
\[
\Prob{ \sup_{x\in  [b(n), \infty)} \bar F_n(x)> c/n} = \Prob{\bar F_n(b(n))> c/n} = \Prob{ \sum_{i=1}^n I(U_i <\bar F(b(n))) > c}.
\]
We can use Lemma 2.3 from \cite{stegehuis2022scale} to upper bound the right-hand side and get 
\begin{equation}
\label{cbtail3}
\Prob{ \sup_{x\in  [b(n), \infty)} \bar F_n(x)> c/n} \leq 
(n\bar F(b(n)))^{\lceil c \rceil}.
\end{equation}
Note that the bounds (\ref{cbtail1}), (\ref{cbtail2}), (\ref{cbtail3}) hold for any continuous distribution $F$ with support on $[0,\infty)$. If $\bar F(x)$ is regularly varying with index $-\alpha<-1$, then 
$a(n)$ is regularly varying with index $1/\alpha -\delta$ and  $b(n)$ is regularly varying with index $1/\alpha +\delta$.
This makes the upper bounds in (\ref{cbtail1}) and (\ref{cbtail2}) go to 0 at a faster rate than polynomial. 

If $\alpha>1$ and $\delta>0$ is such that $1/\alpha +\delta<1$, the upper bound in (\ref{cbtail3}) is regularly varying with index $-\lceil c \rceil \delta \alpha$, which can be made to go to $0$ at any desired polynomial rate by picking $c$ appropriately large. We summarize our findings in the following proposition. 

\begin{proposition}
\label{prop-empiricalconcentration}
Let $\beta>0$. Then 
\begin{equation}
    1-  \Prob{E_n(A,c,\delta)} \leq 2 e^{- n \bar F(a(n)) h(A+1)}  + (n\bar F(b(n)))^{\lceil c \rceil},
\end{equation}
which is $o(n^{-\beta})$ if $\lceil c \rceil > \beta / \alpha \delta$ where $E_n(A,c,\delta)$ is as in~\eqref{eq:En}.
\end{proposition}

\noindent
{\em Application to sample averages.}
To illustrate the use of Proposition \ref{prop-empiricalconcentration},
consider the sample mean 
$\tau_n = \int_0^\infty x d F_n(x)$; in particular the probability 
\begin{equation}
\Prob{ \tau_n > (1+a) \mu}.
\end{equation}
A classical result dating back to Nagaev \cite{Nagaev1969} is that the most likely way this event occurs is by a single big jump of size $a \mu n$
and that this probability is regularly varying with index $-(\alpha-1)$. A critical step in the proof is to show that a jump of size at least $\varepsilon n$ is really necessary. In our notation, this entails showing that
for every $\beta<\infty$ there exists an $\varepsilon>0$ such that
\begin{equation}
\label{rs99}
\Prob{\int_0^{\varepsilon n} x d F_n(x) > \mu +a} = o(n^{-\beta}),
\end{equation} 
or in other words: the probability that a jump of at most $\varepsilon n$ creates the desired deviation is small.
A version of this result (which we actually need in Section \ref{sec:prooftaularge}), can be found in \cite{resnick1999}, and can also be proven with Bennett's inequality. 
We now show how (\ref{rs99}) follows from Proposition \ref{prop-empiricalconcentration}.
On the set $E_n(A,\delta, c)$, the following upper bound holds for $\bar F_n$ on $x\in [0,\varepsilon n]$: 
\begin{align}
\bar F_n(x)  & \leq (1+A) \bar F(x) I(x<a(n)) + (1+A) \bar F(a(n)) I(a(n)\leq x<b(n))\nonumber\\
& \quad + \frac{c}{n} I(x \in [b(n), \varepsilon n]).
\end{align}
Since we integrate against a non-decreasing function, we see that, on $E_n(A,\delta, c)$,
\begin{equation}
\int_0^{\varepsilon n} x d F_n(x) \leq (1+A) \int_0^{\infty} x d F(x) + (1+A)  b(n) \bar F(a(n)) + (c/n) (\varepsilon n).
\end{equation}
This is smaller than $\mu+ c\varepsilon + o(1)$, since $b(n) \bar F(a(n))$ is regularly varying of index 
$1/\alpha +\delta  - \alpha (1/\alpha -\delta)  <0$.
The desired result (\ref{rs99}) now follows by choosing $\varepsilon$ small enough so that $c\varepsilon < a$. 

Below, we apply this bounding technique to the nonlinear functional (\ref{definition-gn}). We believe the technique can 
be applied to other nonlinear functionals of $F_n$, like $U$-statistics, and other observables of inhomogeneous random graphs, such as clustering coefficients, degree correlations or general subgraph counts.

\subsection{Estimating the  number of triangles generated by  large hubs}

To successfully apply Proposition \ref{prop-empiricalconcentration}  to the particular nonlinear functional (\ref{definition-gn}), we need several auxiliary estimates. 
In particular, the following two lemmas will be convenient in the estimation of various single and double integrals appearing in our upper bound of $G_n$, obtained after applying Proposition \ref{prop-empiricalconcentration} to (\ref{definition-gn}). These integrals approximate with high probability the number of additional triangles caused by one or two hubs. 
Their proofs can be found in Appendix \ref{app:sec14appendix}.

\begin{lemma}
\label{lemma-sbc}
There exists a constant $K_1$ such that the following holds. 
Let $1>\alpha_c \geq \alpha_b>1/2$. Let $b(n)$ be regularly varying of index $\alpha_b$ and let $c(n)$ be regularly varying of index $\alpha_c$
with either $\alpha_c>\alpha_b$ or $c(n)= b(n)$. Then 
\begin{equation}\label{eq:Sbc}
S_{b,c}(n) :=  \int_0^\infty f_n(x,b(n),c(n))dF(x) \sim K_1 \frac{b(n)}{c(n)}\bar F(n/ c(n)).
\end{equation}
In particular, $S_{b,c}(n)$ is regularly varying of index  $-[\alpha (1-\alpha_c) + \alpha_c-\alpha_b]$.
\end{lemma}

\begin{lemma}
\label{lemma-sb}
    There exists a constant $K_2$ such that the following holds. 
Let $1>\alpha_b>1/2$. Let $b(n)$ be regularly varying of index $\alpha_b$. Then 
\begin{equation}\label{eq:Sbn}
S_{b}(n) :=  \int_0^\infty \int_x^\infty  f_n(x,y,b(n))dF(y)dF(x) \sim K_2  \frac n {(b(n))^2} \bar F (n/b(n))^2.
\end{equation}
In particular, $S_{b}(n)$ is regularly varying of index  $-[2(\alpha-1)(1-\alpha_b)+1]$.
\end{lemma}

\subsection{$G_n$ cannot be large without a big hub}

In this section, we establish a key result, namely that the following nonlinear analogue of (\ref{rs99}) holds. 
Define $L_n(z)$ as the number of $W_i$ for which $W_i> z$. 

\begin{proposition}
\label{prop-onebigjumpneeded}
There exists an $\varepsilon>0$ such that  
\begin{equation}
    \Prob{G_n > (1+a)m_n ; L_n(\varepsilon c_a(n))=0} = o( n \Prob{W_1> c_a(n)}), 
\end{equation}
where $c_a(n)$ is as in~\eqref{definition-can}.
\end{proposition}

\begin{proof}
On the set $E_n(A,c,d)\cap \{L_n(\varepsilon c_a(n))=0\}$, the following upper bound holds for $\bar F_n$ for all $x\geq 0$
\begin{equation}
\bar F_n(x) \leq \bar F_n^*(x),
\end{equation}
with $\bar F_n^*(x)$ defined by
\begin{equation}
\label{eq-deffnstar}
(1+A) \bar F(x) I(x<a(n)) + (1+A) \bar F(a(n)) I(a(n)\leq x<b(n)) + \frac{c}{n} I(x \in [b(n), \varepsilon c_a(n)]).
\end{equation}
We have that $b(n)/c_a(n)\rightarrow 0$ [for $\delta$ small enough, to be specified later].
Since $f_n$ is nondecreasing in each coordinate for fixed $n$, we can use the property $\bar F_n(x) \leq \bar F_n^*(x)$ to conclude that, on the set $E_n(A, \delta,  c) \cap \{ L_n(\varepsilon c_a(n))=0\}$, 
\begin{equation}
\label{eq:Gnub}
  G_n \leq  n^3 \int_0^\infty \int_x^\infty\int_y^\infty   f_n(x,y,z) dF_n^*(z)dF_n^*(y)dF_n^*(x).
\end{equation}

The next step is to evaluate the integral on the right-hand side of (\ref{eq:Gnub}), where we need to keep track of the value of $x,y,z$ in each of the $3$ terms in $F_n^*$: they may be of (s)mall ($<a(n)$), (m)edium ($b(n)$) or (l)arge ($\varepsilon c_a(n)$) value. 
There are 10 different combinations: (s,s,s), (s,s,m), (s,s,l), (s,m,m), (s,l,l), (s,m,l),  (m,m,m), (m,m,l), (m,l,l), (l,l,l).
Thus, on the set $E_n (A, c, \delta) \cap \{ L_n(\varepsilon c_a(n))=0\}$, 
\begin{align}
\frac{G_n}{(1+A)^3}&\leq  m_n +  n^3\bar F(a(n))\int_0^\infty\int_x^\infty  f_n(x,y,b(n)) dF(y)dF(x) \nonumber \\ 
&+ n^3 \frac{c}{n}\int_0^\infty\int_x^\infty f_n(x,y,\varepsilon c(n)) dF(y)dF(x)\nonumber\\
&  + n^3 (\bar F(a(n)))^2 \int_0^\infty f_n(x,b(n),b(n)) dF(x) \nonumber \\
&+ n^3 (c/n)^2 \int_0^\infty  f_n(x,\varepsilon c_a(n),\varepsilon c_a(n)) dF(x) \nonumber\\
& + n^3 \bar F(a(n)) \frac{c}{n}\int_0^\infty  f_n(x,b(n),\varepsilon c_a(n)) dF(x) \nonumber \\
&+ n^3 (\bar F(a(n)))^3f_n(b(n),b(n),b(n))+n^3 (\bar F(a(n)))^2(c/n)f_n(\varepsilon c(n),b(n),b(n)) \nonumber \\
&+n^3 (\bar F(a(n))) (c/n)^2f_n(\varepsilon c(n),\varepsilon c(n),b(n))+ n^3(c/n)^3f_n(\varepsilon c_a(n),\varepsilon c_a(n),\varepsilon c_a(n)).
\label{eq-tenterms}
\end{align}
Apart from the main term $m_n$, we need to bound 9 terms in total.
In what follows, we often use that $a_n$ is regularly varying of index $1/\alpha -\delta$, that $\bar F(a_n)$ is regularly varying of index $\alpha \delta - 1$,  that $b(n)$ is regularly varying of index $1/\alpha+\delta$, and that $c(n)=c_a(n)$ is regularly varying of index $\alpha_c = \frac 14 \frac{\alpha}{\alpha-1}$. 
The last 4 terms are all bounded by at most $O(n^{3\alpha \delta})$ since $f_n\leq 1$, and are therefore $o(m_n)$.
We now examine Terms 2--6 in more detail. \\

\noindent
{\bf Term 2:} 
Lemma \ref{lemma-sb} with $\alpha_b = 1/\alpha+\delta$ yields that $n^3\bar F(a(n))\int f_n(x,y,b(n)) dF(x)dF(y)$ is regularly varying of index $5-2(\alpha+1/\alpha)+ \delta (3\alpha-2)$. 
For $\delta$ small enough, this can be made strictly smaller than $3-\alpha 3/2$, using the fact that $\alpha + 1/\alpha> 2$ when $\alpha \in (1,2)$. Thus, we can conclude that, for $\delta$ sufficiently small, 
\begin{equation}
n^3\bar F(a(n))\int_0^\infty\int_x^\infty f_n(x,y,b(n)) dF(y)dF(x) = o(m_n).
\end{equation} 

\noindent
{\bf Term 3:}
invoking Lemma \ref{lemma-sb} with $\varepsilon c_a(n)$, and using definition (\ref{definition-can}) we obtain that
\begin{equation}
n^3 \frac{c}{n}\int_0^\infty\int_x^\infty f_n(x,y,\varepsilon c_a(n)) dF(x)dF(y)  \sim c \varepsilon^{(\alpha-1)2} am_n,
\end{equation}
for every $\varepsilon>0$.\\

\noindent
{\bf Term 4:}
invoking Lemma \ref{lemma-sbc} with both sequences equal to $b(n)$, and $\alpha_b= \alpha_c=1/\alpha+\delta$ gives that \\
$n^3 (\bar F(a(n)))^2 \int f_n(x,b(n),b(n)) dF(x)$ is regularly varying of index 
$2-\alpha + \delta 3\alpha$, which is smaller than $3-\alpha 3/2$ for a suitable choice of $\delta$, as $2-\alpha < 3-\alpha 3/2$ for $\alpha <2$. 
Thus, we can conclude that, for $\delta$ sufficiently small, 
\begin{equation}
    n^3 (\bar F(a(n)))^2 \int_0^\infty f_n(x,b(n),b(n)) dF(x) = o(m_n).
\end{equation}

\noindent
{\bf Term 5:}
invoking Lemma \ref{lemma-sbc} with both sequences equal to $\varepsilon c_a(n)$ and  $\alpha_b=\alpha_c= \frac 14 \frac{\alpha}{\alpha-1}$ it follows that\\
$n^3 (c/n)^2 \int_0^\infty f_n(x,\varepsilon c_a(n),\varepsilon c_a(n)) dF(x)$ is regularly varying of index $1-\alpha + \alpha^2 / (4(\alpha-1)$.
This is strictly smaller than $3-\alpha 3/2$: the inequality 
$
1-\alpha + \alpha^2 / (4(\alpha-1) < 3-\alpha 3/2
$
can be rewritten into $(\alpha -4/3)(\alpha+2)<0$ which is true due to our assumption $\alpha \in (4/3,2)$.
Thus, we can conclude that, for $\delta$ sufficiently small, 
\begin{equation}
  n^3 (c/n)^2 \int_0^\infty f_n(x,\varepsilon c_a(n),\varepsilon c_a(n)) dF(x) = o(m_n).
\end{equation}

\noindent
{\bf Term 6:}
invoking Lemma \ref{lemma-sbc} with sequences $b(n)$ and $\varepsilon c_a(n)$, such that $\alpha_b=1/\alpha+\delta$ and $\alpha_c= \frac 14 \frac{\alpha}{\alpha-1}$, it follows that this term behaves like
behaves like Term 5, times an additional factor which is regularly varying of index $\delta (1+\alpha) - (\frac 14 \frac \alpha{\alpha -1} - \frac 1\alpha$).
As this factor converges to $0$ for sufficiently small $\delta$, we conclude also that 
\begin{equation}
  n^3 \bar F(a(n)) \frac{c}{n}\int_0^\infty f_n(x,b(n),\varepsilon c_a(n)) dF(x)=o(m_n).  
\end{equation}
Concluding, we see that, for every $A>0$, on the set $E_n({A, \delta, c}) \cap \{ L_n(\varepsilon c_a(n))=0\}$, 
\begin{equation}
    G_n\leq (1+A) m_n (1+ 3c \varepsilon^{(\alpha-1)2}  a) (1+o(1))
\end{equation}
which is strictly smaller than $1+a$ for $A,\varepsilon$ sufficiently small. We conclude that 
$$\Prob{G_n > (1+a)m_n ;  E_n({A, \delta, c}) \cap \{ L_n(\varepsilon c_a(n))=0\}}=0$$ for sufficiently large $n$, so that 
\begin{align}
& \Prob{G_n > (1+a)m_n ; \{ L_n(\varepsilon c_a(n))=0\}}\nonumber\\
& \quad  \leq \Prob{G_n > (1+a)m_n ; \{ L_n(\varepsilon c_a(n))=0\}; E_n(A,\delta,c)} 
 + \Prob{E_n(A, \delta, c)^c},
\end{align}
which can be made to go to 0 at any polynomial rate by a suitable choice of $c$ and $\varepsilon$, using Proposition \ref{prop-empiricalconcentration}.
\end{proof}

\subsection{Proof of Theorem \ref{thm-gn}}

Using a simple bound for binomial distributions (e.g. ~\cite[Lemma 2.3]{stegehuis2022scale}) one can show that   
\begin{equation}
    \Prob{G_n > (1+a)m_n ; L_n(\varepsilon c_a(n))\geq 2}\leq  \Prob{L_n(\varepsilon c_a(n))\geq 2} = o( n \Prob{W> c_a(n)}). 
\end{equation}
In view of this estimate, Proposition \ref{prop-onebigjumpneeded} and symmetry of $G_n$ as a function of the weights, it suffices to show that
\begin{equation}
\label{eq-singlejobsufficient}
\Prob{G_n> (1+a) m_n, W_n > \varepsilon c_a(n) > W_i, i<n} = (1+o(1))  \Prob{W> c_a(n)}. 
\end{equation}
Write $D_n = \{  W_n > \varepsilon c_a(n) > W_i, i<n \}$ and
\begin{equation}
\label{eq-dnconditioning}
\Prob{G_n> (1+a) m_n, D_n} = \Prob{G_n> (1+a) m_n \mid D_n} \Prob{D_n}. 
\end{equation}
Next, we condition on the  value of $W_n/c_a(n)$ given $D_n$:
\begin{align}
\label{Gncondition}
    & \Prob{G_n> (1+a) m_n \mid D_n} \nonumber\\
    & = \int_{y=\varepsilon}^\infty \Prob{G_n> (1+a) m_n \mid D_n; W_n = yc_a(n)} d\Prob{W_n/c_a(n) \leq y \mid D_n}.
\end{align}
We now state the following proposition, providing a version of the weak law of large numbers for $G_n$, which will be proven later on. 
\begin{proposition}
\label{prop-wlln}
As $n\rightarrow\infty$,
\begin{equation}
\label{eq-ub}
   \Prob{G_n> (1+a) m_n \mid D_n; W_n = yc_a(n)} \rightarrow \begin{cases} 0&  y < 1,\\
   1 & y>1
   \end{cases}
\end{equation}
\end{proposition}
Applying Proposition \ref{prop-wlln} to (\ref{Gncondition}), we see that
\begin{equation}
     \Prob{G_n> (1+a) m_n \mid D_n} \sim \Prob{W_n > c_a(n) \mid D_n},
\end{equation}
which together with (\ref{eq-dnconditioning}) implies (\ref{eq-singlejobsufficient}), proving Theorem \ref{thm-gn}.

\begin{proof}[Proof of Proposition \ref{prop-wlln}]

The idea of the proof is we represent our random graph by a random graph with $n-1$ vertices and truncated weights, and then add a single vertex with weight $yc_a(n)$.  
We first prove (\ref{eq-ub}) for $y>1$.
For $W_i^z=W_iI(W_i<z)$, define the truncated mean $m_{n-1}(z) = \Exp{G_{n-1}(W_1^z,...,W_{n-1}^z)}$, and observe that (e.g.\ by inspecting the proof of Lemma \ref{lemma-mean}), for $a(n) = \bar F^{-1}(n) n^{-\delta}$,
\begin{equation}
\label{eq-truncatedmean}
m_{n-1} (a(n)) = (1+o(1))m_n
\end{equation}
if $\delta>0$ is small enough. 
On the event $D_n$ and when $W_n=yc_a(n)$ we can write, since $a(n) \geq \varepsilon c_a(n)$ for $n$ large enough, and $f_n(x,y,z)$ is non-decreasing in each coordinate for fixed $n$,
\begin{align}
    G_{n}(W_1^{a(n)}, ..., W_{n-1}^{a(n)}, y c_n(a)) &\geq  n^3\int_0^{a(n)} \int_u^{a(n)}\int_v^{a(n)} f_n(u,v,w) dF_{n-1}(w)dF_{n-1}(v)dF_{n-1}(u) \nonumber \\
    &+ n^2\int_0^{a(n)} \int_u^{a(n)}f_n(u,v,yc_a(n)) dF_{n-1}(v)dF_{n-1}(u)
    \label{Gnlowerbound}
\end{align}
To bound the first term on the right-hand side of this expression, 
 observe that, for any constant $A>0$, with high probability, (\ref{cbtail1}) implies that $\bar F_{n-1}(x) \geq (1-A) \bar F(x)$ on $x\in [0,a(n)]$.
 Now, observe that, for any constant $A>0$, with high probability, (\ref{cbtail1}) implies that $\bar F_{n-1}(x) \geq (1-A) \bar F(x)$ on $x\in [0,a(n)]$.  Therefore, we see that, with high probability on the set $D_n$, given $W_n=yc_a(n)$,
\begin{align*}
   n^3\int_0^{a(n)} \int_u^{a(n)}\int_v^{a(n)} f_n(u,v,w) dF_{n-1}(w)dF_{n-1}(v)dF_{n-1}(u) &\geq  (1-A)^3 m_{n-1}(a(n))\\
  &\sim (1-A)^3 m_{n}.
   \end{align*}
To bound the second term in (\ref{Gnlowerbound}), we use again the fact
that $\bar F_{n-1}(x) \geq (1-A) \bar F(x)$ on $x\in [0,a(n)]$ with high probability, to obtain 
\begin{align*}
   & n^2\int_0^{a(n)} \int_u^{a(n)}f_n(u,v,yc_a(n)) dF_{n-1}(v)dF_{n-1}(u)\\
&\geq (1-A)^2n^2\int_0^{a(n)} \int_u^{a(n)}f_n(u,v,yc_a(n)) dF(v)dF(u)\\
&\sim (1-A)^2 n^2 K_2 \frac{2}{(yc_a(n))^2} \bar F\left(\frac{n}{yc_a(n)}\right)^2\\
&\sim (1-A)^2 K_2 n^3 y^{-2}y^{2\alpha} c_a(n)^{2\alpha}\\
&\sim (1-A)^2 K_2 y^{2\alpha-2} am_n.
\end{align*}
In the last 4 steps, 
 we applied Lemma \ref{lemma-sb} with $y c_a(n)$, the defining property of $c_a(n)$, and regular variation.  Since $y>1$, there exists an $A>0$ such that the RHS of (\ref{Gnlowerbound}) is larger
than $(1+a)m_n$ with high probability, in view of (\ref{eq-truncatedmean}), which proves (\ref{eq-ub}) for $y>1$.

We proceed with $y<1$. Fix $y\in (0,1)$ and let $\varepsilon \in (0,y)$. Recall the definition (\ref{eq-deffnstar}) of $F_n^*$. For any $A>0$ we have with high probability on $D_n$
\begin{align}
    G_{n}(W_1^{\varepsilon c_a(a)}, ..., W_n^{\varepsilon c_a(a)}, y c_n(a)) &\leq \mbox{ RHS of (\ref{eq-tenterms}}) \nonumber\\
    & + n^2 \int_0^{\varepsilon c_a(n)} \int_u^{\varepsilon c_a(n)}f_n(u,v,yc_a(z)) 
    dF_n^*(v)dF_n^*(u).
\end{align}
As we have proven, the RHS of (\ref{eq-tenterms}) is close to $m_n$ by taking $\varepsilon$ sufficiently small. It therefore suffices to show that the second term is
strictly smaller than $am_n$ for $n$ sufficientlly large. 
The second term in the right-hand side of the last display can be analyzed in a way similar to (\ref{eq-tenterms}). In particular, we can upper bound it with 
\begin{align*}
     & O(n^{2\alpha \delta}) + n^2 (1+A)^2S_{yc_a(n)}(n)+ n^2(1+A) \frac cn S_{\varepsilon c_a(n), yc_a(n)} \nonumber\\
     & \quad + n^2(1+A) \bar F(a(n)) S_{b(n), yc_a(n)}(n),
\end{align*}
with $S_{b,c}(n)$ and $S_b(n)$ as in~\eqref{eq:Sbc} and~\eqref{eq:Sbn}.
Call the three main terms on the RHS of this expression Term A,B,C. Term B behaves similar to Term 5 in the RHS of (\ref{eq-tenterms}), and Term C behaves like Term 6 in the RHS of (\ref{eq-tenterms}). In particular, both terms are $o(m_n)$. Finally, Term A behaves like $(1+A) m_n a y^{3-\alpha 3/2}$, 
in view of Lemma \ref{lemma-sb} with $y c_a(n)$, the defining property of $c_a(n)$, and regular variation.  
Since $y<1$, there exists an $A>0$ such that the last display is strictly smaller than $am_n$ for $n$ sufficiently large, proving (\ref{eq-ub}) for $y<1$.
\end{proof}

\section{The boundary case $\alpha=4/3$}
\label{sec:gn:alphaboundary}

Recall that, for $b>0$, $X_i^b, i\geq 1$, is an i.i.d.\ sequence such that $\prob(X_i^b>x) = (x/b)^{-\alpha}, x \geq b$.
Set $\eta(a)$ as the smallest number $\eta$ for which  $((k(a)-1)\mu + K_1(\eta) \geq C^3H(1+ a)$. Note that 
$\eta(a)>0$ when  $(k(a)-1) \mu + (k(a)-1)(k(a)-2)/2 < a C^3 H$.

The goal of this section is to prove the following theorem, which serves as a major stepping stone towards Theorem \ref{thm-triangleboundary}. 

\begin{theorem}
\label{thm-gnboundary}
    Suppose that $\Prob{W>x}\sim C x^{-4/3}$ and suppose that  $(k(a)-1) \mu + (k(a)-1)(k(a)-2)/2 <  a C^3 H$.
     Then 
    \begin{equation}
    \label{eq-asymptoticsboundary2}
        \Prob{G_n > (1+a) m_n} \sim
        \prob( K_{k(a)}(X_1^{\eta(a)}, \ldots, X_{k(a)}^{\eta(a)}) \geq C^3Ha) (n\Prob{W> \eta(a) n})^{k(a)}.
    \end{equation}
\end{theorem}

%
%
A key step towards proving Theorem \ref{thm-gnboundary} is to show that $k(a)$ hubs are really needed. Its proof is a refinement of (and builds on) the arguments developed in the case $\alpha>4/3$.

\begin{proposition}
\label{prop-keypropositionintermediatetau}
For every $\beta > 0$ there exists an $\varepsilon >0$ such that
    \begin{equation}
        \Prob{ G_n > m_n(1+a), L_n(\varepsilon n) < k(a)} = o(n^{-\beta}).
    \end{equation}
\end{proposition}

\begin{proof} 
   Fix $l<k(a)$ and observe that $G_n\leq G_{n+l}$.
Note that 
\begin{equation}
\label{eq-gnub0}
    \Prob{ G_{n+l} > m_n(1+a) ; L_n(\varepsilon n) =l} = \binom{n+l}{l}  \Prob{ G_{n+l} > m_n(1+a) ; W_i< \varepsilon n \mbox{ iff } i\leq n}.
\end{equation}
On the event  $\{W_i< \varepsilon n \mbox{ iff } i\leq n\}$ we can write
\begin{align}
\label{eq-boundarydecomposition}
    G_{n+l} & = G_n + n^2\sum_{i=1}^l \int_0^{\varepsilon n} \int_x^{\varepsilon n} f_n(x,y,W_{n+i}) dF_n(y) dF_n(x)\nonumber\\
    & \quad + n\sum_{i=1,j>i}^l \int_0^{\varepsilon n} f_n(x, W_{n+i}, W_{n+j})dF_n(x).
\end{align}
   We wish to use Proposition \ref{prop-onebigjumpneeded} with $\varepsilon c_a(n)$ replaced by $\varepsilon n$ to show that $G_n$ is sufficiently close to $m_n$ with high probability, 
   but in the analysis of Term 5 in (\ref{eq-tenterms}) we assumed that $\alpha>4/3$. 
   For $\alpha=4/3$, this term behaves as 
\begin{align*}
       n^3 (c/n)^2 \int_0^\infty f_n(x,\varepsilon n,\varepsilon n) dF(x) &\sim n c^2  \int_0^\infty \min \{x,\varepsilon/\mu\}^2 d F(x) \leq \varepsilon^2 nc^2/\mu
\end{align*}
which is negligible as $\varepsilon \downarrow 0$. In addition, Term 3 in (\ref{eq-tenterms}) behaves like $n C_1(\varepsilon)$ when $\alpha=4/3$ and $c_a(n)$ is replaced by $\varepsilon n$ which is negligible since $C_1(\varepsilon) \rightarrow 0$ as $\varepsilon \downarrow 0$.
Based on this, and repeating the arguments in the other terms in (\ref{eq-tenterms}) for $\alpha=4/3$ and $c_a(n)$ replaced with $\varepsilon n$, we obtain that
for every $\beta<\infty$ and $\eta>0$ there exists a $c>0$ and $\varepsilon>0$ such that 
\begin{equation}
\label{eq-boundarygnub}
     \Prob{ G_{n} > m_n(1+\eta) ; W_i< \varepsilon n, i\leq n} = o(n^{-\beta}).
\end{equation}
We now analyze the second term in (\ref{eq-boundarydecomposition}).
By bounding $W_{n+i}$ with $\infty$ we get
\begin{equation}
\label{eq-w-infty}
    n^2 \int_0^\infty \int_x^\infty f_n(x,y,W_{n+i}) dF_n(y) dF_n(x) \leq \frac{n}{\mu} \Big(\int_0^{\varepsilon n} xdF_n(x)\Big)^2\leq n\mu.
\end{equation}
Combining this bound with the estimate (\ref{rs99}) implies that for every $\beta<0$ and $\eta>0$ there exists an $\varepsilon>0$ such that
\begin{equation}
\label{eq-boundarygnub2}
\Prob{n^2\sum_{i=1}^l \int_0^{\varepsilon n}\int_x^{\varepsilon n} f_n(x,y,W_{n+i}) dF_n(x) dF_n(y) > n(l\mu + \eta);  W_i< \varepsilon n, i\leq n} = o(n^{-\beta}). 
\end{equation}
The proof is now finished by bounding the last term in (\ref{eq-boundarydecomposition}) by $nl(l-1)/2$, and combining it with (\ref{eq-boundarygnub}), (\ref{eq-boundarygnub2}), and (\ref{eq-gnub0}), noting that
$l \mu + l(l-1)/2 <  a C^3 H $ since $l<k(a)$. 
\end{proof}
Using a simple tail bound for binomial distributions we obtain 
\begin{proposition}
\begin{equation}
     \Prob{ G_n > m_n(1+a), L_n(\varepsilon n) > k(a)} 
     = o(n^{k(a)} \Prob{W>n}^{k(a)}).
\end{equation}
    \end{proposition}
Define
\begin{align}
\label{eq-additional-exp-triangles-multiple}
    G_{n,l}(z_1,...,z_l) & = n^2 \sum_{i=1}^l \int_0^\infty\int_x^\infty  f_n(x,y,nz_i) dF_n(x) dF_n(y) \nonumber\\
    & \quad + n\sum_{i=1,j>i}^l \int_0^\infty f_n(x,nz_i,nz_j) dF_n(x). 
\end{align}
Our next proposition is the final main ingredient of the proof of Theorem \ref{thm-gnboundary}.

\begin{proposition}
\label{prop-hublimit}
The following convergence holds in probability:
    \begin{equation}
        \frac 1n G_{n,l} (z_1,...,z_l) \rightarrow  K_l(z_1,...,z_l).
    \end{equation}
\end{proposition}

\begin{proof} 
Using definition~\eqref{eq:fn}, write the $i$th term of the first part of $ G_{n,l}(z_1,...,z_l)/n$ as 
\begin{align*}
 \int_0^\infty \int_x^\infty  \min \{xy/\mu, n\} \min \{xz_i/\mu, 1\} \min \{yz_i/\mu, 1\}dF_n(y) dF_n(x).
\end{align*}
We  proceed by analyzing upper and lower bounds.
For the upper bound, use $\min \{xy/\mu, n\} \leq xy/\mu$ to get the upper bound 
$(\int_0^\infty x\min \{xz_i/\mu, 1\} dF_n(x) )^2/\mu$. Furthermore,
\begin{align*}
    \int_0^\infty x\min \{xz_i/\mu, 1\} dF_n(x) = \int_{0}^{\mu/z_i} ((z_i/\mu)x^2 -x)dF_n(x) +\int_{0}^\infty x dF_n(x).
\end{align*}
The first integral converges as $F_n\rightarrow F$, and the second integral converges due to the weak law of large numbers. 
The limit equals $(z_i/\mu)\Exp{W^2I(W\leq \mu/z_i)}+\Exp{WI(W>\mu/z_i)}$. This leads to the desired upper bound. 

A lower bound follows by bounding $\min \{xy/\mu, n\}$
from below by $\min \{xy/\mu, K\}$, using that $F_n\rightarrow F$ and Fatou's lemma, and then take $K$ arbitrarily large. 
Finally, note that each term in the second part of  $ G_{n,l}(z_1,...,z_l)/n$ converges to its desired limit using the bounded convergence theorem. 
\end{proof}

\begin{proof}[Proof of Theorem \ref{thm-gnboundary}]
Abbreviate $k=k(a)$.
Using straightforward combinatorial arguments, it suffices to show that 
    \begin{align}
    \label{eq-sufficesboundary}
        & \Prob{G_{n,l}(W_{n+1},...,W_{n+k}) > m_n a,  W_i< \varepsilon n \mbox{ iff } i\leq n}\nonumber\\
        & \quad \sim 
                 \prob( K_{k(a)}(X_1^{\eta(a)}, \ldots, X_{k(a)}^{\eta(a)}) \geq C^3Ha) (\Prob{W> \eta(a) n})^{k(a)}.
    \end{align}
We now write the probability on the LHS of the last display as
\begin{align}
    & \int_{(\varepsilon, \infty)^k} \Prob{  G_{n,l}(z_1,\ldots,z_k) >  m_n a ; W_i<\varepsilon n, i\leq n} d\prod_{i=1}^k \Prob{\frac{W_i}{n} \leq z_i \mid W_i >\varepsilon n} \nonumber \\
    & \times    \Prob{W_1>\varepsilon n}^k  .
\end{align}
Now $\Prob{W_i/n \leq z_i \mid W_i >\varepsilon n}$ converges to the continuous distribution $\Prob{X_i^\varepsilon\leq x_i}$. Recalling that $m_n \sim n C^3 H$, and applying Proposition 
\ref{prop-hublimit},
we obtain that the integral in the last display converges to $\Prob{ K_{k(a)}(X_1^\varepsilon, \ldots, X_{k(a)}^\varepsilon) \geq C^3 H a}$.

Because $(k(a)-1) \mu + (k(a)-1)(k(a)-2)/2 <  a C^3 H$, $K_{k(a)}(\varepsilon, \infty,\ldots,\infty)=0$ for all $\varepsilon<\eta(a)$. Since $K_{k(a)}$ is symmetric, a similar property holds for the other coordinates.
Therefore, if $\varepsilon<\eta(a)$,
\begin{equation}
\prob( K_{k(a)}(X_1^\varepsilon, \ldots, X_{k(a)}^\varepsilon) \geq a) = (\eta/\varepsilon)^{-k(a)\alpha}\prob( C(X_1^{\eta(a)}, \ldots, X_{k(a)}^{\eta(a)}) \geq a).
\end{equation}
Furthermore, by regular variation, 
\begin{equation}
    \prob(W_1>\varepsilon n)^{k(a)} \sim (\eta(a)/\varepsilon)^{k(a)\alpha}  \prob(W_1>\eta(a) n)^{k(a)}.
\end{equation}
Putting everything together, we conclude that (\ref{eq-sufficesboundary}) holds.
\end{proof}

\section{Completing the proofs of Theorem \ref{thm-triangle-singlebigjump} and \ref{thm-triangleboundary}}
\label{sec:prooftaularge}

In this section we use the precise tail asymptotics for $G_n$, $\alpha>4/3$ in Theorem \ref{thm-gn} and for $\alpha=4/3$ in Theorem \ref{thm-gnboundary}, to 
complete the proofs of Theorem \ref{thm-triangle-singlebigjump} and \ref{thm-triangleboundary}. Our argument will be based on the identity 
$G_n = E[ \triangle_n \mid W_1,...,W_n]$, an argument showing that $\triangle_n$ and $G_n$ are close, also in the rare event context we consider. 
Our proof is based on asymptotic upper and lower bounds, which are facilitated by two auxiliary lemmas. The first lemma is helpful for an asymptotic upper bound. 

 \begin{lemma}\label{lem:expconditionedtausmallnew}
	For any $\zeta>0$ there exists some $\varepsilon>0$, such that 
	\begin{equation}
	\Prob{\triangle_n \geq (1+\zeta) G_n} \leq \exp(-n^{\varepsilon}).
		\end{equation}
\end{lemma}
 
We prove this lemma in the next subsection, using a recent concentration bound from~\cite{chatterjee2012missing}.
A crucial argument in the lower bound is to show that large hubs generate sufficiently many additional triangles. This is covered by the next lemma. 

\begin{lemma}
    \label{lem-lb-additionaltriangles}
    Let $\triangle_n(\delta,a)$ be the number of triangles $\{i,j,n\}$ in a graph with $n$ vertices where node $i,j<n$ have weight $W_i I(W_i\leq n^{1/2+\delta})$ and node $n$ weight $c_a(n)$.
The following convergence holds in probability for $\alpha\geq 4/3$, and $\delta>0$ such that $1/2+\delta<1/\alpha$:
\begin{equation}
\label{eq-additionaltrianglessinglehub}
 \triangle_n (\delta,a)/m_n \rightarrow 1+a.
\end{equation}
 Let $\triangle_{n,l}(\delta,z_1,...,z_l)$ be the number of triangles 
 in a graph with $n$ vertices where node $i\leq n-l$ has weight $W_i I(W_i\leq n^{1/2+\delta})$ and node $n-i+1, i=1,...,l$ has weight $nz_i.$
The following convergence holds in probability as $n\rightarrow\infty$ for $\alpha=4/3$, and $\delta>0$ such that $1/2+\delta<3/4$:
\begin{equation}
\label{eq-additionaltrianglesmultiplehubs}
    \triangle_{n,l}(\delta,z_1,...,z_l)/n \rightarrow C^3H+ K_l(z_1,...,z_l),
\end{equation}
with $K_l, H$ as in~\eqref{eq-def-k(a)},~\eqref{eq:h} and $C$ as in Theorem~\ref{thm:smalldevs}.
\end{lemma}

\begin{proof}[Proof of Theorem \ref{thm-triangle-singlebigjump} and Theorem \ref{thm-triangleboundary}.]
The proofs of both theorems are similar. We first prove an asymptotic upper bound. Write for $0<\zeta<1$,
\begin{align*}
\Prob{\triangle_n > (1+a)m_n} &\leq \Prob{\triangle_n\geq (1+\zeta) G_n} + \Prob{ G_n/(1+\zeta) \geq \triangle_n \geq (1+a)m_n}\\
&\leq e^{-n^\varepsilon} + \Prob{G_n > (1+(a-\zeta)(1-\zeta))m_n}.
\end{align*}
Consequently, 
\[
\limsup_{n\rightarrow\infty} \frac{\Prob{\triangle_n > (1+a)m_n}}{\Prob{G_n > (1+a)m_n}}\leq 
\limsup_{n\rightarrow\infty} \frac{\Prob{G_n > (1+(a-\zeta)(1-\zeta))m_n}}{\Prob{G_n > (1+a)m_n}}
\]
The RHS of this expression converges to $1$ as $\zeta\downarrow 0$: for $\alpha>4/3$ this follows from Theorem \ref{thm-gn}, and for $\alpha=4/3$ this follows from Theorem \ref{thm-gnboundary},
in particular from (\ref{eq-asymptoticsboundary2}): under our assumptions, $k(\cdot)$ is continuous at $a$, and so are $\eta(a)$ and $\prob( K_{k(a)}(X_1^{\eta(a)}, \ldots, X_{k(a)}^{\eta(a)}) \geq C^3Ha)$.

We proceed with the proof of an asymptotic lower bound. 
First, consider $\alpha>4/3$, and write, for sufficiently large $n$
\begin{align*}
    \Prob{\triangle_{n} > (1+a)m_n} &\geq n \Prob{W_n > c_{a+\zeta}(n)} \Prob{\triangle_{n}(\delta,a+\zeta) > (1+a)m_n}.
\end{align*}
By Lemma \ref{lem-lb-additionaltriangles}, $\Prob{\triangle_{n}(\delta,a+\eta) > (1+a)m_n}\rightarrow 1$ for sufficiently small $\delta$,
 and we obtain
\begin{equation}
    \liminf_{n\rightarrow\infty} \frac{\Prob{\triangle_{n} > (1+a)m_n}}{\Prob{G_{n} > (1+a)m_n}} \geq \liminf_{n\rightarrow\infty}\frac{\Prob{W> c_{a+\zeta}(n)}}{\Prob{W> c_a(n)}}.
\end{equation}
The RHS converges to 1 as $\zeta\downarrow 0$ due to the  properties of $c_a(n)$ in Lemma \ref{lemma-can}.

Next, consider  $\alpha=4/3$. Write, for $\zeta>0$ and sufficiently large $n$
\begin{align*}
    \Prob{\triangle_{n} > (1+a)m_n} &\geq n^{k(a)} \Prob{K_{k(a)}(W_1/n,...,W_{k(a)}/n) > C^3H(a+\zeta)}\\
    & \hspace{1cm} \cdot \Prob{\triangle_{n} > (1+a)m_n \mid K_{k(a)}(W_1/n,...,W_{k(a)}/n) > C^3H(a+\zeta)}.
\end{align*}
By conditioning on $(W_1/n,...,W_{k(a)}/n)$ and applying the second part of Lemma \ref{lem-lb-additionaltriangles}, we see that (using 
 $m_n \sim C^3Hn$)
$\Prob{\triangle_{n} > (1+a)m_n \mid K_{k(a)}(W_1/n,...,W_{k(a)}/n) > C^3H(a+\zeta) }\rightarrow 1$. 
Since, for sufficiently small $\zeta>0$, $k(a)=k(a+\zeta)$,  we have that 
\begin{align}& \Prob{K_{k(a)}(W_1/n,...,W_{k(a)}/n) > C^3H(a+\zeta)} \nonumber\\& \sim 
 \prob( K_{k(a)}(X_1^{\eta(a+\zeta)}, \ldots, X_{k(a)}^{\eta(a+\zeta)}) \geq C^3H(a+\zeta)) (\Prob{W> \eta(a+\zeta) n})^{k(a)}.\end{align}
Consequently, using Theorem \ref{thm-gnboundary},
\begin{align}
      & \liminf_{n\rightarrow\infty} \frac{\Prob{\triangle_{n} > (1+a)m_n}}{\Prob{G_{n} > (1+a)m_n}} \nonumber\\
      & \geq \liminf_{n\rightarrow\infty}
      \frac{\prob( K_{k(a)}(X_1^{\eta(a+\zeta)}, \ldots, X_{k(a)}^{\eta(a+\zeta)}) \geq C^3H(a+\zeta)) (\Prob{W> \eta(a+\zeta) n})^{k(a)}}{\prob( K_{k(a)}(X_1^{\eta(a)}, \ldots, X_{k(a)}^{\eta(a)}) \geq C^3Ha) (\Prob{W> \eta (a) n})^{k(a)}}.
\end{align}
Since $\eta(\cdot)$ is continuous at $a$, the RHS converges to $1$ as $\zeta\downarrow 0$.
\end{proof}

\subsection{Proof of Lemma \ref{lem:expconditionedtausmallnew}}

The number of triangles $\triangle_n$ equals the sum of the indicators that $i,j,k$ forms a triangle over all $i,j,k$. In the proofs, we will often make use of a recent concentration bound from~\cite{chatterjee2012missing} to deal with the dependencies of the presences of different triangles, which we state here for completeness:
\begin{lemma}[Theorem~3.1 from~\cite{chatterjee2012missing}]\label{lem:chatterjeelem}
    Let $F$ be a finite set and $(X_i)_{i\in F}$ , $(X'_i)_{i\in F}$ $(X_{i(j)})_{i,j\in F}$ be
collections of nonnegative random variables with finite moment generating functions, defined on the same probability space, and satisfying the following conditions:
\begin{enumerate}[a)]
    \item  For all $i$, $X_i\leq  X'_i$.
    \item For all $i$, the random variables $X'_i$ and $\sum_{j\in F}X_{j(i)}$ are independent.
    \item For all $i$, $\sum_{j\in F}X_{j(i)}\leq \sum_{j\in F}X_{j}$.
    \item There is a constant $a$ such that for all $i$, when $X_i > 0$, we have
    \begin{equation}
        \sum_{j\in F}X_{j}\leq a+ \sum_{j\in F}X_{j(i)}.
    \end{equation}
\end{enumerate}
Let $\lambda = \sum_{j\in F}\Exp{X'_{j}}$. Then for any $t \geq\lambda$,
\begin{equation}
    \Prob{\sum_{i\in F}X_i\geq t}\leq\exp\Big(-\frac{t}{a}\log\Big(\frac{t}{\lambda}-1+\frac{\lambda}{t}\Big)\Big).
\end{equation}
\end{lemma}
We will apply this lemma to the number of triangles, where $X_{ijk(uvw)}$ deals with the dependence of the events that the triangle $ijk$ is present and the event that the triangle $uvw$ is present. To carry out this idea, we need several additional supporting results.

The second preparatory lemma shows that edges between vertices of low weights appear in relatively few triangles:
\begin{lemma}\label{lem:lowdegtriangleedges}
 Suppose that $1<\alpha<2$. Then, when 
 $K> n^{(2-\alpha)/2+\varepsilon}$ for some $\varepsilon>0,$
 \begin{equation}
     \Prob{\{i,j\} \text{ in }\geq K \text{ triangles} \mid W_iW_j<\mu n }\leq \exp(-c_1K),
 \end{equation}
 for some $c_1:=c_1(\alpha,\delta)>0$ and $n$ sufficiently large. 
\end{lemma}
The proof of this lemma is based on  bounding the number of triangles for low-weight vertices from above with a binomial random variable with the right probability, and can be found in Appendix~\ref{app:sec14appendix}.


Define the event $E_w = \{W_1=w_1,...,W_n=w_n\}$ and set $g_n= \sum_{i<j<k} f_n(w_i,w_j,w_k)$ as the expected number of triangles conditional on specific values of the weights. Let $\triangle_n(w)$ be the conditional probability of the number of triangles conditionally on $E_w$.
Finally, set for $\zeta>0$,
\begin{equation}
    \label{eq-def-J}
    J(\eta) = (1+\zeta) \log (\zeta + 1/(1+\zeta)) /3.
\end{equation}
\begin{lemma}
\label{lem-conditionedchatterjeetriangles} 
There exists an $\varepsilon>0$ such that, for $\zeta>0$, and all $w=(w_1,...,w_n)$:
\begin{equation}
     \Prob{\triangle_n(w) >(1+\zeta )g_n} \leq e^{-J(\zeta) g_n/n^{2-\alpha}}+ e^{-n^\varepsilon}.
\end{equation}
  \end{lemma}
  
\begin{proof}[Proof of Lemma \ref{lem-conditionedchatterjeetriangles}]
By Lemma~\ref{lem:lowdegtriangleedges}, when $w_iw_j<\mu n$, and choosing $K=n^{2-\alpha}$, 
\begin{equation}\label{eq:ijtriangles}
	\Prob{ \{i,j\} \text{ in }\geq n^{2-\alpha} \text{ triangles}}\leq \exp(-K_1 n^{2-\alpha}),
\end{equation} 
for some $K_1>0.$
This indicates that with probability at least $1-n^2\exp(-K_1 n^{2-\alpha}) \geq 1- \exp(-n^\varepsilon)$, all edges between vertices of weights $w_iw_j<\mu n$ are in at most $n^{2-\alpha}$ triangles.
We now work on this event, which we call $\mathcal{E}'$.

We set $X_{ijk}=X'_{ijk}$ as the indicator that a triangle is present between vertices $i,j,k$. Furthermore, we set $X_{ijk(uvw)}=X_{ijk}$ when $|\{i,j,k,u,v,w\}|\geq 5$. When $|\{i,j,k,u,v,w\}|= 4$, we set $X_{ijk(uvw)}=X_{ijk}$ when the overlap of $ijk$ and $uvw$ occurs at an edge with $w_uw_v>\mu n$ (where we assumed w.l.o.g. that the overlap occurs at $u,v$), and we set $X_{ijk(uvw)}= 0$ otherwise or when $\{i,j,k\}=\{u,v,w\}$. Then, $\sum_{i,j,k}X_{ijk(uvw)}$ and $X_{uvw}$ are independent. Indeed, when two triangles do not overlap at an edge, their presence is independent conditionally on the weights, as the edge indicators are independent conditionally on the weights. When the edge overlap occurs at an edge that is present with probability one, the presence of the two triangles is still independent conditionally on the weights. In all other cases, $X_{ijk(uvw)}=0$, which is also independent of $X_{uvw}$.

On the event $\mathcal{E}'$,
\begin{equation}
	\triangle_n =\frac{1}{6}\sum_{i,j,k}X_{ijk}\leq \sum_{ijk}X_{ijk(uvw)}+3n^{2-\alpha}.
\end{equation}
 Lemma~\ref{lem:chatterjeelem} with $a=3n^{2-\alpha}$, $t=(1+\zeta)g_n$ and $\lambda=g_n$ concludes the assertion. 
\end{proof}

\begin{proof}[Proof of Lemma \ref{lem:expconditionedtausmallnew}]
	Recall that $F_n(x)$ denotes the empirical weight distribution. 
Now by (\ref{eq-tailconcentration}), the event 
$\mathcal{E}_1= \{ \sup_{x< \sqrt{\mu n}} |\frac{\bar F_n(x)}{\bar F(x)}-1| \leq \eta \}$ happens with probability of at least $1-e^{-n^\varepsilon}$ for
some $\varepsilon>0$. Fix some $a<1$. On the event $\mathcal{E}_1$, 
\begin{align}\label{eq:exptrianglb}
    G_n& \geq n^3 \int_{a\sqrt{\mu n}}^{\sqrt{\mu n}}\int_{a\sqrt{\mu n}}^{\sqrt{\mu n}}\int_{a\sqrt{\mu n}}^{\sqrt{\mu n}}f_n(u,v,w) dF_n(u)dF_n(v)dF_n(w)\nonumber\\
    & = \mu^{-3}\Big(\int_{a\sqrt{\mu n}}^{\sqrt{\mu n}}w^2dF_n(w)\Big)^3 \geq  \varepsilon_2 n^{3-3\alpha/2}L(\sqrt{n})^3
\end{align}
for some $\varepsilon_2>0.$
Now, write
 \begin{align}
	\Prob{\triangle_n >(1+\zeta )G_n} \leq \Prob{\mathcal{E}_1^c} + 	\Prob{\triangle_n >(1+\zeta )G_n; \mathcal{E}_1}.
\end{align}
The first term is asymptotically small. The second term can be bounded by applying Lemma \ref{lem-conditionedchatterjeetriangles},
using the lower bound (\ref{eq:exptrianglb}) for $G_n$ on $\mathcal{E}_1$ to conclude the assertion.  
\end{proof}

\subsection{Proof of Lemma \ref{lem-lb-additionaltriangles}}

To prove (\ref{eq-additionaltrianglessinglehub}), we wish to apply Chebyshev's inequality, which requires appropriate estimates for the first
two moments of $\triangle_n (\delta,a)$.
Note first that
\begin{align*}
    \Exp{ \triangle_n (\delta,a)} &= (n-1)^3\int_{0<x<y<z<n^{1/2+\delta}} f_{n}(x,y,z) dF(x)dF(y)dF(z)\\
    & \hspace{1cm}+(n-1)^2\int_{0<x<y<n^{1/2+\delta}} f_{n}(x,y,c_a(n)) dF(x)dF(y).
\end{align*}
Using this expression, and applying similar ideas as in the proofs of Lemma \ref{lemma-mean} and Lemma \ref{lemma-sb}, along with the definition 
 of $c_a(n)$ it follows that $\Exp{ \triangle_n (\delta,a)} \sim (1+a)m_n$.

We proceed by analyzing the second moment of $\triangle_n (\delta,a)$ by using 
the concentration bound Lemma \ref{lem-conditionedchatterjeetriangles}, as well as a concentration bound for $F_n$. In particular, it follows from (\ref{cbtail1}) that for sufficiently small $\delta>0$ there exists an $\varepsilon>0$ such that 
\begin{equation}
\label{eq-tailconcentration}
\Prob{\sup_{x< n^{1/2+\delta}} |\frac{\bar F_n(x)}{\bar F(x)}-1| > \zeta} \leq e^{-n^\varepsilon}.
 \end{equation}
On the event  $E_n(\zeta) = \{\sup_{x< n^{1/2+\delta}} |\frac{\bar F_n(x)}{\bar F(x)}-1| \leq  \zeta\}$, 
$G_{n-1}/n^{2-\alpha} > n^\varepsilon$ for some $\varepsilon>0$ for $n$ sufficiently large (see also the detailed computation in the proof of 
Lemma \ref{lem-conditionedchatterjeetriangles}). In addition, using the expression
 \begin{align*}
 \triangle_n (\delta,a) &= (n-1)^3\int_{0<x<y<z<n^{1/2+\delta}} f_{n-1}(x,y,z) dF_{n-1}(x)dF_{n-1}(y)dF_{n-1}(z)\\
    & \hspace{1cm}+(n-1)^2\int_{0<x<y<n^{1/2+\delta}} f_{n-1}(x,y,c_a(n)) dF_{n-1}(x)dF_{n-1}(y),
\end{align*}
we see that $\triangle_n (\delta,a) < (1+\zeta)^3(1+a)m_n$ with high probability on the event $E_n$, while outside of this event $\triangle_n (\delta,a)\leq n^3$, and therefore
\begin{equation}
    \Exp{ \triangle_n (\delta,a)^2} \leq n^6 (2e^{-n^\varepsilon} + e^{-J(\zeta) n^\varepsilon}) + (1+\zeta)^4 (1+a)^2m_n^2.
\end{equation}
Next, fix $\delta>0$. Using Chebyshev's inequality and the previous bounds we obtain for every $\zeta>0$,
\begin{align*}
\limsup_{n\rightarrow\infty}\Prob{|\triangle_n(\delta, a)/\Exp{ \triangle_n (\delta,a)}- 1| > \delta} &\leq \frac 1{\delta^2} ( (1+\zeta)^4-1)).
\end{align*} 
The proof of  (\ref{eq-additionaltrianglessinglehub}) is now completed by letting $\zeta \downarrow 0$.
 
We now turn to the proof of (\ref{eq-additionaltrianglesmultiplehubs}) 
Note that
\begin{align*}
    \Exp{\triangle_{n,l}(\delta,z_1,...,z_l)} &= (n-l)^3\int_{0<x<y<z<n^{1/2+\delta}} f_{n-l}(x,y,z) dF(x)dF(y)dF(z)\\
    &
    \hspace{1cm}+(n-l)^2 \sum_{i=1}^l \int_{0<x<y<n^{1/2+\delta}}  f_{n-l}(x,y,(n-l)z_i) dF(x) dF(y)\\
    &\hspace{1cm}+ (n-l) \sum_{i=1,j>i}^l \int_0^{n^{1/2+\delta}} f_n(x,(n-l)z_i,(n-l)z_j) dF(x). 
\end{align*}
Observe that the second and third term in this expression are very related to the quantities analyzed in Proposition \ref{prop-hublimit}.
Using similar ideas as in the proof of Proposition \ref{prop-hublimit} with $F_n$ replaced by $F$, it follows from the above expression that 
\begin{equation}
     \Exp{\triangle_{n,l}(\delta,z_1,...,z_l)}/n\rightarrow C_l(z_1,...,z_l).
\end{equation}
We proceed by analyzing the second moment of $\triangle_{n,l}(\delta,z_1,...,z_l)$, using concentration bounds. 
As before, note that on the event  $E_n(\zeta) = \{\sup_{x< n^{1/2+\delta}} |\frac{\bar F_n(x)}{\bar F(x)}-1| \leq  \zeta\}$, 
$G_{n-l}/n^{2-\alpha} > n^\varepsilon$ for some $\varepsilon>0$ for $n$ sufficiently large.

Modify the definition of $G_{n-l,l}$ in Proposition \ref{prop-hublimit} to $G_{n-l,l}^{n^{1/2+\delta}}$ to truncate all integrals at 
$n^{1/2+\delta}$ rather than $\infty$, i.e. 
\begin{align}
\label{eq-additional-exp-triangles-multiple-truncated}
    G_{n,l}^{n^{1/2+\delta}}(z_1,...,z_l) & = n^2 \sum_{i=1}^l \int_{0<x<y<n^{1/2+\delta}}  f_n(x,y,nz_i) dF_n(x) dF_n(y)\nonumber\\
    & \quad + n\sum_{i=1,j>i}^l \int_0^{n^{1/2+\delta}} f_n(x,nz_i,nz_j) dF_n(x). 
\end{align}
Again using Lemma \ref{lem-conditionedchatterjeetriangles}, we obtain that, with high probability, 
\begin{equation}
    \triangle_{n,l}(\delta,z_1,...,z_l) \leq (1+\zeta) (G_{n-l} +  G_{n,l}^{n^{1/2+\delta}}(z_1,...,z_l)),
\end{equation}
which is in turn bounded by $(1+\zeta)^2 n (C^3 H + C_l(z_1,...,z_l))$ with high probability. The proof is now completed by using Chebyshev's inequality, and letting $\zeta\downarrow 0$ as before.


\section{Many dominating hubs}\label{sec:prooftausmall}

To prove Theorems~\ref{thm:smalldevs} and~\ref{thm:triangldptausmall}, we distinguish several types of vertices, based on the vertex weights. Specifically, fix $\zeta>2$ and $\varepsilon>0$ and let $\gamma$ as in Theorem~\ref{thm:triangldptausmall}, then
\begin{itemize}
	\item Type A: $W_i\leq \sqrt{\mu n}$,
	\item Type $B_3$: $W_i\leq D^{1/\alpha}n^{(3-\gamma)/(2\alpha)}\log(n^{(\gamma-3)/2+\alpha})^{-1/\alpha}$, 
	\item Type $B_2$: $W_i\in [D^{1/\alpha}n^{(3-\gamma)/(2\alpha)}\log(n^{(\gamma-3)/2+\alpha})^{-1/\alpha},n/\log(n)^{\zeta/\alpha}]$,
	\item Type $B_1$: $W_i\geq n/\log(n)^{\zeta/\alpha}$,
\end{itemize}
where we set
\begin{equation}
    D = \frac{\sqrt{2a}}{h(1+\varepsilon)},
\end{equation}
with $h(\cdot)$ as in~\eqref{eq:h}.
Now depending on $\gamma$ and $\alpha$, vertices of type A and $B_3$ or type $A$ and $B_2$ may overlap, see Figure~\ref{fig:absets}. In the proof of Theorem~\ref{thm:smalldevs}, we will therefore sometimes split up the sets $B_3$ or $B_2$ at $\sqrt{\mu n}$ to avoid this overlap.
We denote the number of triangles between one vertex from $B_{i}$ one from $B_j$ and one from $B_k$ by $\triangle_{B_i,B_j,B_k}$. Similarly, we denote the number of edges with one end in $B_i$ and one end in $B_j$ by $E_{B_i,B_j}$. The main strategy of the proof of Theorem~\ref{thm:triangldptausmall} is to split up the triangles into different types, and bounds their contributions. The main contribution is from triangles of type $B_1B_1B_i$ for $i=1,2,3$, that is, triangles with two high-degree vertices and one other vertex. We show that all other types of triangles appear less often with high enough probability. 
Here we will use that the empirical weight distribution of $B_3$ vertices is close to its mean with sufficiently high probability, to get rid of the randomness caused by the weight sampling. On $B_2$, we get rid of the random weights by using the fact that the probability that a triangle appears is non-decreasing in the weights, so that we may assume that all $B_2$ vertices have weights $n/\log(n)^{\zeta/\alpha}.$


\begin{figure}
    \centering
    \begin{subfigure}[b]{0.45\linewidth}
        \includegraphics[width=\linewidth]{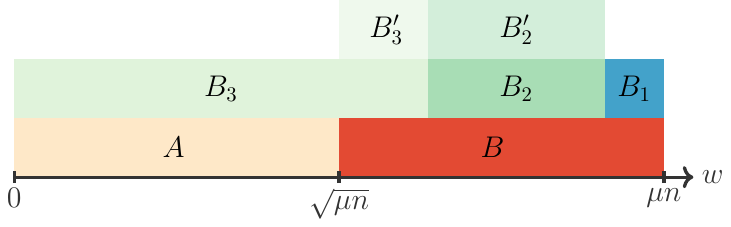}
        \caption{Case 1: $A\subseteq B_3$}
        \label{fig:case1}
    \end{subfigure}
    \hfill
    \begin{subfigure}[b]{0.45\linewidth}
        \includegraphics[width=\linewidth]{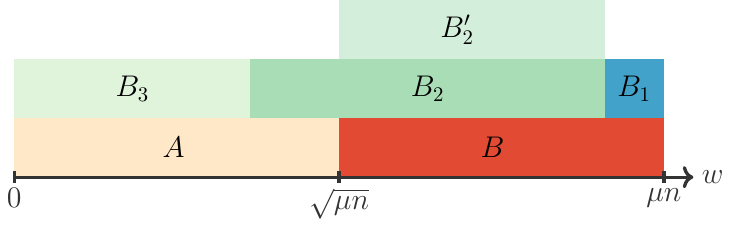}
        \caption{Case 2: $B_3\subseteq A$.}
    \end{subfigure}
    \caption{Illustration of the sets $A$, $B$ and $B_1$, $B_2$, $B_3$}
    \label{fig:absets}
\end{figure}

Before we prove Theorem \ref{thm:smalldevs}, we first provide several lemmas. We begin by recalling a variation of 
Theorem A.1.4 in \cite{AlonSpencer}, to bound tail probabilities of sums of independent Bernoulli random variables. 
\begin{lemma}
\label{lem-concentration}
Let $B_i,i\geq 1$ be a sequence of independent Bernoulli random variables with $p_i=\prob(B_i=1)=1-\prob(B_i=0)$.
Set $m_n = \sum_{i=1}^n p_i$. For every $b>0$ we have 
\begin{equation}
    \prob(\sum_{i=1}^n B_i > (1+b) m_n) \leq e^{-m_n I_B(b)}, \hspace{1cm} \prob(\sum_{i=1}^n B_i < (1-b) m_n) \leq e^{-m_n I_B(-b)},
\end{equation}
with $I_B(b) = (1+b) \log (1+b)-b$.
\end{lemma}
We next provide an elementary lemma that bounds the probability that polynomially many vertices have at least a given weight:
\begin{lemma}\label{lem:nvertices}
Suppose that $\Prob{W>x} \sim C x^{-\alpha}$. Then, for $\gamma>1-\alpha\beta$ and $d,u>0$,
    \begin{equation}
        \Prob{un^\gamma \text{ vertices of weight }>dn^\beta} \leq  \exp\Big(-un^{\gamma} \log(n^{\gamma-1+\alpha\beta})\Big)(1+o(1)),
    \end{equation}
    and 
    \begin{equation}
        \Prob{un^\gamma \text{ vertices of weight }>dn^\beta} \geq  \frac{1}{\sqrt{2n}}\exp\Big(-un^{\gamma} \log(n^{\gamma-1+\alpha\beta})\Big)(1+o(1)).
    \end{equation}
\end{lemma}
\begin{proof}
By~\eqref{D-tail}, the probability that a vertex has weight at least $dn^\beta$ is given by $Cd^{-\alpha}n^{-\alpha\beta}(1+o(1))$, and is independent for each vertex. Thus, by Lemma~\ref{lem-concentration},
\begin{align}
& \Prob{n^\gamma \text{ vertices of weight }>cn^\beta}\nonumber\\
& \leq  \exp\Big(-un^{\gamma} \log\Big(\frac{un^\gamma}{Cn^{1-\alpha\beta}d^{-\alpha}}\Big)+n^\gamma-Cd^{-\alpha}n^{-\alpha\beta}\Big)(1+o(1))\nonumber\\
& \leq \exp\Big(-un^{\gamma} \log(n^{\gamma-1+\alpha\beta})\Big)(1+o(1)).
 \end{align}
 The second inequality follows similarly, using~\cite[Lemma 4.7.2]{ash2012information} instead to get the lower bound. 
\end{proof}
We now provide bounds on the number of vertices in $B_1$ and $B_2$, $N(B_1)$ and $N(B_2$)  (we will usually upper bound the number of $B_3$ vertices by the total number of vertices $n$). 

\begin{lemma}
Suppose that $\Prob{W>x} \sim C x^{-\alpha}$. For $K\gg n^{(\gamma-1)/2}\log(n^{(\gamma-3)/2+\alpha})$,
    \begin{equation}\label{eq:Nb1b2}
        \Prob{N(B_1\cup B_2)>K}\leq \exp\Big(-K\log\Big(\frac{K}{D^{-\alpha}n^{(\gamma-1)/2}\log(n^{(\gamma-3)/2+\alpha})}\Big)\Big)(1+o(1)),
    \end{equation}
    and for $K\gg n^{1-\alpha}\log(n)^\zeta $,
    \begin{equation}\label{eq:Nb1}
        \Prob{N(B_1)>K}\leq \exp\Big(-K\log\Big(\frac{K}{n^{1-\alpha}\log(n)^\zeta}\Big)\Big)(1+o(1)).
    \end{equation}
\end{lemma}
\begin{proof}
    As the number of vertices of weight at least $x\gg 1$, $N_x$ is binomial with parameters $n$ and $Cx^{-\alpha}(1+o(1))$, Lemma~\ref{lem-concentration} gives that 
    \begin{align}
        \Prob{N_x>K} & \leq \exp\Big(-K\log\Big(\frac{K}{Cnx^{-\alpha}}\Big)+K-1\Big)(1+o(1)) \nonumber\\
        & \leq \exp\Big(-K\Big(\log\Big(\frac{K}{Cnx^{-\alpha}}\Big)-1\Big)\Big)(1+o(1)).
    \end{align}
    Plugging in the lower weight bounds for $B_1\cup B_2$ and for $B_1$, $n^{(3-\gamma)/(2\alpha)}\log(n^{(\gamma-3)/2+\alpha})^{-1/\alpha}$ and $n/\log(n)^{\zeta/\alpha}$ respectively and noticing that under the given constraints on $K$ the $-1$ in the exponent is of lower order of magnitude, gives the result.
\end{proof}

\subsection{Bounding specific triangle and edge types}
We now turn to investigating the number of edges and triangles between vertices of the groups $A$, $B_1,B_2,B_3$, that will later be used in the upper bound for the total number of triangles. 
We first provide some lemmas that bound the number of edges and triangles between $B_3$ vertices:
\begin{lemma}\label{lem:b3triangedges}
Suppose that $\Prob{W>x} \sim C x^{-\alpha}$. Let $B_3$ be the set of vertices with weights at most 
 \begin{equation}\label{eq:Qn}
 	Q_n=D^{-1/\alpha}n^{(3-\gamma)/(2\alpha)}\log(n^{(\gamma-3)/2+\alpha})^{-1/\alpha}.
 \end{equation}
 Then, for $\gamma\in(1,3)$
 \begin{align}\label{eq:triangb3b3b3}
 	\Prob{\triangle_{B_3,B_3,B_3}>Kn^\gamma}
 	& \leq n\exp\Big(-\sqrt{2a}n^{(\gamma-1)/2}\log(n^{(\gamma-3)/2+\alpha}))\Big)
 \end{align}
when $Kn^\gamma >(1+\varepsilon)^3C^3Hn^{3-3\alpha/2}$. 
 Furthermore, for $K>(1+\varepsilon)^2\mu n,$
 \begin{equation}\label{eq:edb3b3}
 	\Prob{E_{B_3,B_3}>K}\leq 
  \exp(-Dn^{(\gamma-1)/2}\log(n^{(\gamma-3)/2+\alpha})h(1+\varepsilon)).
 \end{equation}
\end{lemma}
\begin{proof}
	By~\eqref{cbtail1}, for all $\varepsilon>0$ and with $Q_n$ as in~\eqref{eq:Qn},
	\begin{align}
		\Prob{\mathcal{E}_1}&:=\Prob{\sup_{x\in[1, Q_n]}\frac{|1-F_n(x)|}{\bar{F}(x)}\leq (1+\varepsilon)}\nonumber\\
  & \geq 1- \exp\big(-Dh(1+\varepsilon) n^{(\gamma-1)/2}\log(n^{(\gamma-3)/2+\alpha})\big)(1+o(1)).
	\end{align}
Given any weight distribution satisfying the event $\mathcal{E}_1$, the edge count is a sum of independent Bernoulli random variables with mean at most $(1+\varepsilon)\mu n$. Thus, on $\mathcal{E}_1$, Lemma~\ref{lem-concentration} yields that for $K>(1+\varepsilon)^2\mu n$,
\begin{align}
		\Prob{E_{B_3,B_3}\geq K \mid \mathcal{E}_1 }& \leq \Prob{E_{B_3,B_3}\geq K \mid \mathcal{E}_1} \nonumber\\
  & \leq \exp\Big(-(1+\varepsilon)\mu n((1+\varepsilon)\log(1+\varepsilon)-\varepsilon)\Big)\nonumber\\
  & \leq \exp(-\sqrt{2a}n^{(\gamma-1)/2}\log(n^{(\gamma-3)/2+\alpha})),
\end{align}
for $\gamma<3$ and $n$ sufficiently large.
To obtain an upper bound for the triangle counts, we apply Lemma~\ref{lem:chatterjeelem}. Let $X_{ijk}=X'_{ijk}$ denote the indicator that $i,j,k$ forms a triangle. Now given the weight distribution, $X_{ijk}$ and $X_{uvw}$ are independent as long as $|\{i,j,k,u,v,w\}|\geq 5$. Thus, we define $X_{ijk(uvw)}=X_{ijk}$ when $|\{i,j,k,u,v,w\}|\geq 5$ and $X_{ijk(uvw)}=0$ otherwise. Then, for all $uvw$,
\begin{equation}\label{eq:dub3bound}
	\sum_{i,j,k,\in B_3}X_{ijk}\leq  d_{u,B_3}+d_{v,B_3}+d_{w,B_3}+\sum_{i,j,k,\in B_3}X_{ijk(uvw)}
\end{equation}
where $d_{u,B_3}$ denotes the degree of vertex $u$ to other $B_3$ vertices.
As $d_{u,B_3}\leq n$ for all $u$, on $\mathcal{E}'$, by~\eqref{eq:dub3bound} and Lemma~\ref{lem:chatterjeelem},
\begin{align}
	\Prob{\triangle_{B_3,B_3,B_3}\geq Kn^\gamma}& \leq \exp\Big(-\frac{Kn^\gamma}{3n}\log\Big(\frac{Kn^\gamma}{(1+\varepsilon)^3H n^{3-3\alpha/2}}-1+\frac{(1+\varepsilon)^3H n^{3-3\alpha/2}}{Kn^\gamma}\Big)\Big)\nonumber\\
 &   \leq \exp(-Dh(1+\varepsilon)n^{(\gamma-1)/2}\log(n^{(\gamma-3)/2+\alpha}))
\end{align}
for $\gamma\geq 1$ and $n$ sufficiently large, as on $\mathcal{E}_1$, the mean number of triangles is bounded by $(1+\varepsilon)^3Hn^{3-3\alpha/2}(1+o(1))$. Plugging in the value of $D$ then proves the lemma.
\end{proof}

We now investigate vertices and triangles from $B_2$ vertices. The proof of this Lemma follows a similar structure as the proof of Lemma~\ref{lem:b3triangedges}, and can be found in Appendix~\ref{app:manyhubs}.
\begin{lemma}\label{lem:b2edgetriang}
	Suppose that $\Prob{W>x} \sim C x^{-\alpha}$ and that $d_u\leq M$ for all $u\in A$. Then,
	\begin{equation}
		\Prob{E_{A,B_2}>n^{(\gamma+1)/2}\mid N(B_2)}\leq \exp\Bigg(-\frac{n^{(\gamma+1)/2}}{M} \log\Big(\frac{n^{(\gamma-1)/2}}{K_1 N(B_2)\log(n)^{-\zeta}}\Big)\Bigg),
	\end{equation}
for some $K_1>0$,
where $N({B_2})$ denotes the number of $B_2$ vertices. 
Furthermore,
\begin{equation}
	\Prob{ \triangle_{A,B_2,B_2}> n^\gamma\mid N(B_2)} \leq \exp\Big(-\frac{n^\gamma}{M^2}\log\Big(\frac{n^{\gamma-1}}{3K_1 N(B_2)^2\log(n)^{-\zeta}}\Big)\Big).
\end{equation}
\end{lemma}

To prove Theorem~\ref{thm:triangldptausmall}, we finally state a lemma that bounds the degree of type $A$ vertices. Its proof can also be found in Appendix~\ref{app:manyhubs}.
\begin{lemma}\label{lem:Adegrees}
Suppose that $\Prob{W>x} \sim C x^{-\alpha}$.	Let $\mathcal{F}$ denote the event that all vertices of weight at most $ \sqrt{\mu n}$ have degrees at most
	\begin{equation}
	    M=\begin{cases}
	    4\sqrt{\mu n} & \gamma<2\nonumber\\
	    4n^{\gamma/3} & \gamma\in[2,3).
	    \end{cases}
	\end{equation}
	Then, for all $D>0$,
	\begin{equation}
		\Prob{\bar{\mathcal{F}}}\leq \exp\Big(-Dn^{(\gamma-1)/2} \log(n^{(\gamma-3)/2+\alpha})\Big).
	\end{equation}
\end{lemma}

\subsection{Proof of Theorem~\ref{thm:smalldevs}}\label{sec:proofmoderate}
\begin{proof}[Proof of Theorem~\ref{thm:smalldevs}] To prove this theorem, we lower bound the number of triangles with triangles between two weight $n$ vertices and one other vertex, and we upper bound by considering triangles between vertices of weights in classes $B_1, B_2, B_3$, and split into all possible cases, as different types of triangles can be upper bounded by different terms. 

 \textit{Lower bound.}
 As a lower bound of $\triangle_n$, we compute the number of triangles with at least one vertex of weight $> \mu n$. Note that as the minimal weight is at least 1, vertices of weight $\mu n$ connect to all other vertices with probability 1. The probability that at least $\sqrt{2 a}n^{1-\alpha3/4}$ vertices of weights at least $\mu n$ are present can be bounded by Lemma~\ref{lem:nvertices} as
\begin{align}
    \Prob{\sqrt{2 a}n^{1-\alpha3/4} \text{ vertices of weight }>\mu n}
    & \geq \frac{1}{\sqrt{2n}}\exp\Big(-{\sqrt{2 a}n^{1-\alpha3/4}}\log\big(n^{\alpha/4}\big)\Big).
\end{align}
Now, $\sqrt{2 a}n^{1-\alpha3/4}$ vertices of weight at least $\mu n$ generate $n(\sqrt{2 a}n^{1-\alpha3/4})^2/2=a n^{3-3\alpha/2}$ triangles with probability one. Indeed, every pair of 2 vertices of weight at least $\mu n$ forms a triangle with any of the other $n$ vertices, creating $a n^{3-3\alpha/2}$ triangles. Furthermore, the vertices of weight at most $n^{1/\alpha}$ generate $C^3Hn^{3-3\alpha/2}(1+o(1))$ 
triangles with high probability~\cite{hofstad2017b}. As these triangles use different edges, their presence is independent from the presence of the triangles with the $\mu n$ weight vertices.
Therefore, 
\begin{equation}\label{eq:trianglbdetailed}
     \Prob{\triangle_n>n^{3-3\alpha/2}(C^3H+a)}\geq  \frac{1}{\sqrt{2n}}\exp\Big(-{\sqrt{2 a}n^{1-\alpha3/4}}\log\big(n^{\alpha/4}\big)\Big)(1+o(1)).
\end{equation} 
Thus,
\begin{equation}
    \liminf_{n\to\infty}\frac {\log  \Prob{\triangle_n>n^{3-3\alpha/2}(C^3H+a)}}{n^{1-\alpha3/4}\log(n)}
    \geq \sqrt{2 a}\frac{-\alpha}{4}.
\end{equation}

\textit{Upper bounds.}
When $\gamma=3-3\alpha/2$, $B_3$ vertices have weights up to $Dn^{3/4}\log(n^{\alpha/4})^{1/\alpha}$. We denote $B'_3=B_3\cap B$. Then, $B_3'\subseteq B$ and $A\subseteq B_3$, as illustrated in Figure~\ref{fig:case1}. Therefore, we distinguish the types of triangles $B_3B_3B_3$, $B_3B_3B_i$, $B_1B_1B_i$, $B_3'B_iB_j$,  A$B_iB_j$, $B_iB_jB_k$ for $i,j,k\in\{1,2\}$ and bound the number of these triangles one by one. 

\textit{$B_3B_3B_3$ triangles. }
For $B_3B_3B_3$ triangles, we use Lemma~\ref{lem:b3triangedges} with $D=\sqrt{2 a}/h(1+\varepsilon)$ to obtain that for $\delta>\varepsilon$,
\begin{equation}\label{eq:b3deltasmall}
    \Prob{\triangle_{B_3B_3B_3}>(C^3H+a)n^{3-3\alpha/2}}\leq \exp\Big(-{\sqrt{2 a}n^{1-\alpha3/4}}\log\big(n^{\alpha/4}\big)\Big).
\end{equation}

\textit{$B_3B_3B_i$ triangles for $i \in \{1,2\}$. }
We bound the number of $B_3B_3B_i$ triangles for $i=1,2$ by the number of $B_i$ vertices times the number of edges between $B_3$ vertices. 
By~\eqref{eq:Nb1b2} with $\gamma=3-3\alpha/2$ and $D =\sqrt{2 a}/h(1+\varepsilon)$,
\begin{equation}\label{eq:Nb1b2A}
    \Prob{N(B_1\cup B_2)>K}\leq \exp\Big(-K\log\Big(\frac{K}{\sqrt{2 a}h(1+\varepsilon)^{-1}n^{1-\alpha3/4}\log(n^{\alpha/4})}\Big)\Big).
\end{equation}
By~\eqref{eq:edb3b3} with $D=\sqrt{2 a}$
\begin{equation}
    \Prob{E_{B_3,B_3}> n^{1-3\alpha/8}}\leq \exp(-\sqrt{2 a}n^{(\gamma-1)/2}\log(\sqrt{2 a} n^{(\gamma-3)/2+\alpha})).
\end{equation}
Then, by~\eqref{eq:Nb1b2A},
\begin{align}\label{eq:triangab2b3deltasmall}
    \Prob{\triangle_{B_3B_3B_i}>\varepsilon n^{3-3\alpha/2}\mid E_{B_3,B_3}< n^{1-3\alpha/8}}& \leq \Prob{N(B_1\cup B_2)>\varepsilon n^{2-9\alpha/8}}\nonumber\\
    & \leq \exp\Big(-\varepsilon n^{2-9\alpha/8} \log(\delta^{-\alpha/2}n^{2-9\alpha/8-1-\alpha3/4}))\log(n^{\alpha/4})^{-1}\Big)\nonumber\\
    & \leq \exp(-\sqrt{2 a}n^{(\gamma-1)/2}\log(n^{(\gamma-3)/2+\alpha}))
\end{align}
for $n$ sufficiently large and $\alpha<4/3$. 

\textit{$B_iB_1B_1$ triangles for $i\in\{1,2,3\}$. }
We now bound the number of $B_iB_1B_1$ type triangles $i\in\{1,2,3\}$ by $n$ times the number of pairs of two $B_1$ vertices. This is the number of ways to choose 2 type $B_1$ vertices, and one other vertex. Thus, by~\eqref{eq:Nb1},
\begin{align}\label{eq:triangab1b1deltasmall}
    \Prob{\triangle_{B_iB_1B_1}>an^{3-3\alpha/2}}& \leq \Prob{N(B_1)>\sqrt{2 a}n^{1-\alpha3/4}}\nonumber\\
    & \leq \exp\Big(-\sqrt{2 a}n^{1-\alpha3/4}\log\Big(\frac{\sqrt{2 a}n^{1-\alpha3/4}}{n^{1-\alpha}\log(n)^\zeta}\Big)\Big).
\end{align}
Thus, 
\begin{equation}
    \limsup_{n\to\infty}\frac{\log\Big(\Prob{\triangle_{B_iB_1B_1}>an^{3-3\alpha/2}}\Big)}{n^{1-\alpha3/4}\log(n)}\leq -\sqrt{2a}(1-\frac{3\alpha}{4})+(1-\alpha) = \sqrt{2a}\frac{-\alpha}{4}.
\end{equation}

\textit{$B_iB_jB_k$ triangles for $i,j,k\in\{1,2\}$.}
Any triple of vertices in $B_1\cup B_2$ have weights at least $\sqrt{\mu n}$ and thus form a triangle with probability one. Thus, an upper bound for these triangles is the number of vertices in $B_2\cup B_1$ to the power three. 
\begin{align}
&\Prob{\triangle_{B_iB_jB_k}>\varepsilon n^{(2-\alpha)3/2}} \leq     \Prob{N(B_2\cup B_1)^3 >\varepsilon n^{(2-\alpha)3/2}  }\nonumber\\
& = \Prob{N(B_2\cup B_1) >\varepsilon^{1/3} n^{(2-\alpha)/2} }.
\end{align}
Now by~\eqref{eq:Nb1b2A} 
\begin{align}
	& \Prob{N(B_2\cup B_1)>\varepsilon^{1/3} n^{(2-\alpha)/2}}\nonumber\\
 &\leq \exp\Big(-\varepsilon^{1/3} n^{(2-\alpha)/2}\log\Big(\frac{\varepsilon^{1/3} n^{(2-\alpha)/2}}{\sqrt{2 a}h(1+\varepsilon)^{-1}n^{1-\alpha3/4}\log(n^{\alpha/4})}\Big)\Big)\nonumber\\
	& \leq \exp\Big(-\sqrt{2 a} n^{1-\alpha3/4} \log(n^{\alpha/4})\Big),
\end{align}
when $\zeta>1$ and $n$ is sufficiently large,
since $(2-\alpha)/2>1-\alpha3/4$ for $1<\alpha<4/3$.
Thus, for $\varepsilon, \delta>0$, and $n$ sufficiently large,
\begin{equation}\label{eq:triangbbbdeltasmall}
    \Prob{\triangle_{B_iB_jB_k}>\varepsilon n^{3-3\alpha/2}}\leq \exp\Big(-{\sqrt{2 a}n^{1-\alpha3/4}}\log\big(n^{\alpha/4}\big)\Big)
\end{equation}
for $i,j,k\in\{1,2\}$.

\textit{$B_3'B_iB_j$ triangles $i,j\in\{1,2\}$.}
Again, any triple of vertices in $B_3'B_iB_j$ have weights at least $\sqrt{\mu n}$ and thus form a triangle with probability one. Thus, an upper bound for these triangles is the number of vertices of weight at least $\sqrt{\mu n}$ squared times the number of vertices in $B_2\cup B_1$. 
\begin{align}
&\Prob{\triangle_{B_3'B_iB_j}>\varepsilon n^{(2-\alpha)3/2}} \leq     \Prob{N(B_2\cup B_1)N(\geq \sqrt{\mu n})^2 >\varepsilon n^{(2-\alpha)/2}  },
\end{align}
where $N(\geq \sqrt{\mu n})$ denotes the number of vertices with weight at least $\sqrt{\mu n}$.
Now by~\eqref{eq:Nb1b2A} 
\begin{align}\label{eq:Nb2detailed}
	\Prob{\mathcal{F}}:=\Prob{N(B_2\cup B_1)>n^{1-\alpha3/4}\log(n)^{\zeta}}
	& \leq \exp\Big(-\sqrt{2 a} n^{1-\alpha3/4} \log(n^{\alpha/4})\Big),
\end{align}
when $\zeta>1$ and $n$ is sufficiently large. 
On the complement of this event, Lemma~\ref{lem:nvertices} yields
\begin{align}
&\Prob{\triangle_{B_3'B_iB_j}>\varepsilon n^{(2-\alpha)3/2}\mid \bar{\mathcal{F}}} 
 \leq \Prob{N(B)^2 >\varepsilon n^{1-3\alpha/8}\log(n)^{-\zeta}\mid \bar{\mathcal{F}}}\nonumber\\
    & \leq \exp\Big(-\varepsilon n^{1-3\alpha/8}\log(n)^{-\zeta}\log(\frac{\varepsilon n^{1-3\alpha/8}\log(n)^{-\zeta}}{n^{(2-\alpha)/2}})\Big)\nonumber\\
    & \leq \exp\Big(-{\sqrt{2 a}n^{1-\alpha3/4}}\log\big(n^{\alpha/4}\big)\Big),
\end{align}
since $1-3/8\alpha>1-\alpha3/4$ for $1<\alpha<4/3$.
Thus, for $\varepsilon, \delta>0$, and $n$ sufficiently large,
\begin{equation}\label{eq:triangbbb3deltasmall}
    \Prob{\triangle_{B_3'B_iB_j}>\varepsilon n^{3-3\alpha/2}}\leq 2\exp\Big(-{\sqrt{2 a}n^{1-\alpha3/4}}\log\big(n^{\alpha/4}\big)\Big).
\end{equation}

\textit{A$B_2B_2$ triangles.}
We now consider A$B_2B_2$ triangles.
By Lemma~\ref{lem:b2edgetriang}, for all $\varepsilon>0$
\begin{align}\label{eq:b2deltasmall}
    \Prob{ \triangle_{AB_2B_2}> \varepsilon n^{3-3\alpha/2}}& \leq \exp\Big(-\varepsilon \frac{n^{3-3\alpha/2}}{K^2n}\log\Big(\frac{\varepsilon n^{3-3\alpha/2}}{3K_1 N(B_2)^2n\log(n)^{-\zeta}}\Big)\Big)\nonumber\\
    & \leq \exp\Big(-\sqrt{2 a}n^{2-3/2\alpha} \log(n^{\alpha/4})\Big),
\end{align}
for $n$ sufficiently large when $N(B_2)<Cn^{1-\alpha3/4}\log(n)^{\zeta/2}$ for some $C>0$.
Now by~\eqref{eq:Nb2detailed}
\begin{align}
    \Prob{N(B_2)>n^{1-\alpha3/4}\log(n)^{\zeta/2}}
     & \leq \exp\Big(-\sqrt{2 a} n^{1-\alpha3/4} \log(n^{\alpha/4})\Big),
\end{align}
for $\zeta>2$.

\textit{$AB_2B_1$ triangles.}
We now bound the number of $AB_2B_1$ triangles by $N(B_1)$ times the number of $AB_2$ edges. 
By Lemma~\ref{lem:b2edgetriang} with $M=\max(\sqrt{n},n^{\gamma/3})$,
\begin{equation}
    \Prob{E_{A,B_2} > n^{1-3\alpha/8}}
    \leq \exp\Big(-\sqrt{2 a} n^{1-\alpha3/4} \log(n^{\alpha/4})\Big)
\end{equation}
as long as $N(B_2)\leq n^{1-\alpha3/4}\log(n)^{\zeta}$, which happens with probability~\eqref{eq:Nb2detailed}.
Thus, by~\eqref{eq:Nb1},
\begin{align}\label{eq:triangab2b1deltasmall}
    \Prob{\triangle_{AB_1B_2}>\varepsilon n^{3-3\alpha/2}}& \leq \Prob{N(B_1)E_{AB_2}>\varepsilon n^{3-3\alpha/2}}\nonumber\\
    & \leq \Prob{N(B_1)>\varepsilon n^{2-9\alpha/8}}\nonumber\\
    & \leq \exp\Big(-\sqrt{2 a} n^{1-\alpha3/4} \log(n^{\alpha/4})\Big),
\end{align}
for $\alpha<4/3$ and $\varepsilon>0$. 





Thus,~\eqref{eq:b3deltasmall},~\eqref{eq:triangab2b3deltasmall},~\eqref{eq:triangbbbdeltasmall},~\eqref{eq:triangbbb3deltasmall} ~\eqref{eq:b2deltasmall} and~\eqref{eq:triangab2b1deltasmall} yield that
\begin{align}
    & \Prob{\sum_{\{i,j,k\}\neq \{3,3,3\}, \{i,1,1\}}\triangle_{B_iB_jB_k}>\varepsilon n^{3-3\alpha/2}} 
    \leq K \exp\Big(-\sqrt{2 a} n^{1-\alpha3/4} \log(n^{\alpha/4})\Big),
\end{align}
for any $\varepsilon>0$ and some $K>0$. Combining this with~\eqref{eq:b3deltasmall} and~\eqref{eq:triangab1b1deltasmall}yields that for $\delta>\varepsilon$,
\begin{align}
\Prob{\triangle_n>(H+a) n^{3-3\alpha/2}}& \leq \Prob{\triangle_{B_3B_3B_3}+\triangle_{B_1B_1B_i}>(C^3H+a-\varepsilon) n^{3-3\alpha/2}}\nonumber\\
& \quad +K \exp\Big(-\sqrt{2 a} n^{1-\alpha3/4} \log(n^{\alpha/4})\Big)\nonumber\\
    & \leq (K+1) \exp\Big(-\sqrt{2 a} n^{1-\alpha3/4} \log(n^{\alpha/4})\Big) \nonumber\\
    & \quad + \Prob{\triangle_{B_iB_1B_1}>(C^3H+a-2\varepsilon) n^{3-3\alpha/2}}.
\end{align}
Thus,
\begin{equation}\label{eq:triangabbfinaldeltasmall}
     \limsup_{n\to\infty}\frac{\log\Big(\Prob{\triangle_n>(C^3H+a) n^{3-3\alpha/2}}\Big)}{n^{1-\alpha3/4}\log(n)}\leq \sqrt{2a-2\varepsilon}\Big(\frac{-\alpha}{4}\Big).
\end{equation}
Letting $\varepsilon\downarrow 0$ yields the result. 
\end{proof}

\paragraph{Acknowledgements.} C. Stegehuis was supported by NWO VENI grant 202.001 and NWO M2 grant 0.379.


\bibliographystyle{abbrv}
\bibliography{references,references2}

\newcommand{\noop}[1]{}
\begin{thebibliography}{10}

\bibitem{AlonSpencer}
N.~Alon and J.~H. Spencer.
\newblock {\em The probabilistic method}.
\newblock Wiley Series in Discrete Mathematics and Optimization. John Wiley \&
  Sons, Inc., Hoboken, NJ, fourth edition, 2016.

\bibitem{andreis2021}
L.~Andreis, W.~König, H.~Langhammer, and R.~I.~A. Patterson.
\newblock A large-deviations principle for all the components in a sparse
  inhomogeneous random graph.
\newblock {\em arXiv:2111.13200}.

\bibitem{ash2012information}
R.~B. Ash.
\newblock {\em Information theory}.
\newblock Courier Corporation, 2012.

\bibitem{Augeri2020}
F.~Augeri.
\newblock Nonlinear large deviation bounds with applications to {W}igner
  matrices and sparse {E}rdos-{R}\'{e}nyi graphs.
\newblock {\em Ann. Probab.}, 48(5):2404--2448, 2020.

\bibitem{ayan}
A.~Bhattacharya, Z.~Palmowski, and B.~Zwart.
\newblock Persistence of heavy-tailed sample averages: principle of infinitely
  many big jumps.
\newblock {\em Electron. J. Probab.}, 27:Paper No. 50, 25, 2022.

\bibitem{bgt87}
N.~H. Bingham, C.~M. Goldie, and J.~L. Teugels.
\newblock {\em Regular variation}, volume~27 of {\em Encyclopedia of
  Mathematics and its Applications}.
\newblock Cambridge University Press, Cambridge, 1987.

\bibitem{boguna2003}
M.~Bogu\~n\'a and R.~Pastor-Satorras.
\newblock Class of correlated random networks with hidden variables.
\newblock {\em Phys. Rev. E}, 68:036112, 2003.

\bibitem{chakrabarty2021}
A.~Chakrabarty, R.~S. Hazra, F.~den Hollander, and M.~Sfragara.
\newblock Large deviation principle for the maximal eigenvalue of inhomogeneous
  erd{\H{o}}s-r{\'{e}}nyi random graphs.
\newblock {\em Journal of Theoretical Probability}, 35(4):2413--2441.

\bibitem{chatterjee2012missing}
S.~Chatterjee.
\newblock The missing log in large deviations for triangle counts.
\newblock {\em Random Structures \& Algorithms}, 40(4):437--451, 2012.

\bibitem{ChatterjeeDembo2016}
S.~Chatterjee and A.~Dembo.
\newblock Nonlinear large deviations.
\newblock {\em Adv. Math.}, 299:396--450, 2016.

\bibitem{Litvak2017}
N.~Chen, N.~Litvak, and M.~Olvera-Cravioto.
\newblock Generalized {P}age{R}ank on directed configuration networks.
\newblock {\em Random Structures Algorithms}, 51(2):237--274, 2017.

\bibitem{chung2002}
F.~Chung and L.~Lu.
\newblock The average distances in random graphs with given expected degrees.
\newblock {\em Proc. Natl. Acad. Sci. USA}, 99(25):15879--15882, 2002.

\bibitem{Collamore2018}
J.~F. Collamore and S.~Mentemeier.
\newblock Large excursions and conditioned laws for recursive sequences
  generated by random matrices.
\newblock {\em Ann. Probab.}, 46(4):2064--2120, 2018.

\bibitem{demarco2011}
B.~DeMarco and J.~Kahn.
\newblock Upper tails for triangles.
\newblock {\em Random Structures {\&} Algorithms}, 40(4):452--459.

\bibitem{dhara2019}
S.~Dhara, R.~van~der Hofstad, and J.~S.~H. van Leeuwaarden.
\newblock Critical percolation on scale-free random graphs: New universality
  class for the configuration model.
\newblock {\em Communications in Mathematical Physics}, 382(1):123--171, feb
  2021.

\bibitem{dommers2018}
S.~Dommers, C.~Giardin{\`{a}}, C.~Giberti, and R.~van~der Hofstad.
\newblock Large deviations for the annealed ising model on inhomogeneous random
  graphs: Spins and degrees.
\newblock 173(3-4):1045--1081.

\bibitem{EmbrechtsGoldie}
P.~Embrechts and C.~M. Goldie.
\newblock On closure and factorization properties of subexponential and related
  distributions.
\newblock {\em J. Austral. Math. Soc. Ser. A}, 29(2):243--256, 1980.

\bibitem{fosskorshunov2012}
S.~Foss and D.~Korshunov.
\newblock On large delays in multi-server queues with heavy tails.
\newblock {\em Mathematics of Operations Research}, 37(2):201--218, 2012.

\bibitem{friedrich2015a}
T.~Friedrich and A.~Krohmer.
\newblock Parameterized clique on inhomogeneous random graphs.
\newblock {\em Discrete Appl. Math.}, 184:130--138, 2015.

\bibitem{heydari2017}
H.~Heydari and S.~M. Taheri.
\newblock Distributed maximal independent set on inhomogeneous random graphs.
\newblock In {\em 2017 2nd Conference on Swarm Intelligence and Evolutionary
  Computation ({CSIEC})}. {IEEE}, 2017.

\bibitem{hofstad2017b}
R.~{\swap{Hofstad}{van der }}, A.~J. E.~M. Janssen, J.~S.~H. van Leeuwaarden,
  and C.~Stegehuis.
\newblock Local clustering in scale-free networks with hidden variables.
\newblock {\em Phys. Rev. E}, 95(2):022307, 2017.

\bibitem{hofstad2017}
R.~{\swap{Hofstad}{van der }}, P.~van~der Hoorn, N.~Litvak, and C.~Stegehuis.
\newblock Limit theorems for assortativity and clustering in null models for
  scale-free networks.
\newblock {\em Advances in Applied Probability}, 52(4):1035–1084, 2020.

\bibitem{hofstad2017d}
R.~{\swap{Hofstad}{van der }}, J.~S.~H. van Leeuwaarden, and C.~Stegehuis.
\newblock Optimal subgraph structures in scale-free configuration models.
\newblock {\em The Annals of Applied Probability}, 31(2), 2021.

\bibitem{janson2004}
S.~Janson, K.~Oleszkiewicz, and A.~Ruci{\'{n}}ski.
\newblock Upper tails for subgraph counts in random graphs.
\newblock {\em Israel Journal of Mathematics}, 142(1):61--92, dec 2004.

\bibitem{janson2004a}
S.~Janson and A.~Ruci{\'{n}}ski.
\newblock The deletion method for upper tail estimates.
\newblock {\em Combinatorica}, 24(4):615--640.

\bibitem{janssen2019}
A.~J. E.~M. Janssen, J.~S.~H. van Leeuwaarden, and S.~Shneer.
\newblock Counting cliques and cycles in scale-free inhomogeneous random
  graphs.
\newblock {\em Journal of Statistical Physics}, 175(1):161--184, feb 2019.

\bibitem{kerriou2022}
C.~Kerriou and P.~Mörters.
\newblock The fewest-big-jumps principle and an application to random graphs.
\newblock 2022.

\bibitem{kim2004}
J.~H. Kim and V.~H. Vu.
\newblock Divide and conquer martingales and the number of triangles in a
  random graph.
\newblock {\em Random Structures and Algorithms}, 24(2):166--174, 2004.

\bibitem{Mikosch2013}
T.~Mikosch and O.~Wintenberger.
\newblock Precise large deviations for dependent regularly varying sequences.
\newblock {\em Probability Theory Related Fields}, 156(3-4):851--887, 2013.

\bibitem{Nagaev1969}
A.~V. Nagaev.
\newblock Limit theorems that take into account large deviations when
  {C}ram\'{e}r's condition is violated.
\newblock {\em Izv. Akad. Nauk UzSSR Ser. Fiz.-Mat. Nauk}, 13(6):17--22, 1969.

\bibitem{oliveira2019}
R.~I. Oliveira and G.~H. Reis.
\newblock Interacting diffusions on random graphs with diverging average
  degrees: Hydrodynamics and large deviations.
\newblock {\em Journal of Statistical Physics}, 176(5):1057--1087.

\bibitem{Mariana2021}
M.~Olvera-Cravioto.
\newblock Page{R}ank's behavior under degree correlations.
\newblock {\em Ann. Appl. Probab.}, 31(3):1403--1442, 2021.

\bibitem{resnick1999}
S.~Resnick and G.~Samorodnitsky.
\newblock Activity periods of an infinite server queue and performance of
  certain heavy tailed fluid queues.
\newblock {\em Queueing Systems Theory Appl.}, 33(1-3):43--71, 1999.
\newblock Queues with heavy-tailed distributions.

\bibitem{rheeblanchetzwart2016}
C.-H. Rhee, J.~Blanchet, and B.~Zwart.
\newblock Sample path large deviations for lévy processes and random walks
  with regularly varying increments.
\newblock {\em Annals of Probability}, 6(47):3551--3605, 2019.

\bibitem{stegehuis2017b}
C.~Stegehuis.
\newblock Degree correlations in scale-free random graph models.
\newblock {\em Journal of Applied Probability}, 56(3):672–700, 2019.

\bibitem{stegehuis2019b}
C.~Stegehuis, R.~van~der Hofstad, and J.~S.~H. van Leeuwaarden.
\newblock Variational principle for scale-free network motifs.
\newblock {\em Scientific Reports}, 9(1):6762, 2019.

\bibitem{stegehuis2022scale}
C.~Stegehuis and B.~Zwart.
\newblock Scale-free graphs with many edges.
\newblock {\em arXiv:2212.05907}, 2022.

\bibitem{vazquez2002}
A.~V{\'a}zquez, R.~Pastor-Satorras, and A.~Vespignani.
\newblock Large-scale topological and dynamical properties of the internet.
\newblock {\em Phys. Rev. E}, 65:066130, 2002.

\bibitem{Wellner1978}
J.~A. Wellner.
\newblock Limit theorems for the ratio of the empirical distribution function
  to the true distribution function.
\newblock {\em Z. Wahrsch. Verw. Gebiete}, 45(1):73--88, 1978.

\bibitem{yao2017}
D.~Yao, P.~van~der Hoorn, and N.~Litvak.
\newblock Average nearest neighbor degrees in scale-free networks.
\newblock {\em Internet Mathematics}, 2018.

\bibitem{yin2019a}
H.~Yin, A.~R. Benson, and J.~Ugander.
\newblock Measuring directed triadic closure with closure coefficients.
\newblock {\em Network Science}, 8(4):551–573, 2020.

\bibitem{zwart2004}
B.~Zwart, S.~Borst, and M.~Mandjes.
\newblock Exact asymptotics for fluid queues fed by multiple heavy-tailed
  on–off flows.
\newblock {\em The Annals of Applied Probability}, 14(2):903--957, May 2004.

\end{thebibliography}



\appendix

\section{Additional proofs for Sections \ref{sec:intro}--\ref{sec:prooftaularge}}

\label{app:sec14appendix}

\begin{proof}[Proof of Lemma \ref{lemma-mean}]
Let $$h(u,v,w)= \begin{cases}
    \frac{8uvw}{\mu^3} & uv, vw, uw<\mu \\
    \frac{2v}{\mu^2} & uv, vw<\mu , uw\geq \mu \\
    \frac{2u}{\mu2} & uv, uw<\mu , vw\geq \mu \\
    \frac{2w}{\mu^2} & uw, vw<\mu , uv\geq \mu \\
    0 & \text{else}.
\end{cases}$$
Then, we can write 
\begin{align*}
    \frac{6m_n }{n^3 \bar F(\sqrt{n})^3 } &= \frac {1}{\bar F(\sqrt{n})^3} \int_{x=0}^\infty\int_{y=0}^\infty\int_{z=0}^\infty f_n(x,y,z) dF(x)dF(y)dF(z)\\
    &= \int_{u=0}^\infty\int_{v=0}^\infty\int_{w=0}^\infty f_1(u,v,w) d\hat F_n(u)d \hat F_n(v)d\hat F_n(w)\\
    &= \int_{u=0}^\infty\int_{v=0}^\infty\int_{w=0}^\infty h(u,v,w) \frac{\bar F(u\sqrt{n})}{\bar F(\sqrt{n})} \frac{\bar F(v\sqrt{n})}{\bar F(\sqrt{n})} \frac{\bar F(w\sqrt{n})}{\bar F(\sqrt{n})} dudvdw,
\end{align*}
where we made the transformation $u=x/\sqrt{n}$ (and similar for $(v,w)$) in the third step, with $\hat F_n(u) = F(\sqrt{n} u) / \bar F(\sqrt{n})$.
To show this integral converges we use the Potter bounds, which imply that for each $\delta>0$ there exists a constant $M$ such that 
$\frac{\bar F(u\sqrt{n})}{\bar F(\sqrt{n})} \leq M u^{-\alpha-\delta}, u<1$ and $\frac{\bar F(u\sqrt{n})}{\bar F(\sqrt{n})} \leq M u^{-\alpha+\delta}, u>1$.
Define the function $d(u)= M u^{-\alpha-\delta} I(u\leq 1) + M u^{-\alpha+\delta} I(u>1)$. Since for $\alpha \in (1,2)$, the integral 
\[
\int_{v=0}^\infty\int_{w=0}^\infty h(u,v,w) d(u) d(v)d(w) dudvdw
\]
converges for $\delta$ sufficiently small, we can use dominated converge
to conclude 
\begin{align*}
        \frac{m_n }{n^3 \bar F(\sqrt{n})^3 } &\rightarrow 
        \frac 16 \int_{u=0}^\infty\int_{v=0}^\infty\int_{w=0}^\infty h(u,v,w) u^{-\alpha} v^{-\alpha}w^{-\alpha}dudvdw\\
        &=\frac{\alpha^3}{6} \int_{u=0}^\infty\int_{v=0}^\infty\int_{w=0}^\infty f_1(u,v,w)  u^{-\alpha-1} v^{-\alpha-1} w^{-\alpha-1} dudvdw.
\end{align*}
\end{proof}

\begin{proof}[Proof of Lemma \ref{lemma-can}]
Using Lemma \ref{lemma-sb} we know that, for a sequence $b(n)$ regularly varying of index $\alpha_b$,
\begin{equation}
n^2\int_0^\infty \int_0^\infty  f_n(x,y,b(n))dF(x)dF(y) \sim K_2  \frac {n^3} {(b(n))^2} \bar F (n/b(n))^2.
\end{equation}
Now, choose $b$ such that $K_2  \frac {n^3}{(b(n))^2} \bar F (n/b(n))^2 = a m_n$. 
Dividing both sides of this equation with $n$, and noting that $y \bar F(y)$ has an asymptotic inverse $h(y)$ which is regularly varying of index $-1/(\alpha-1)$ (cf.\ \cite{bgt87}, Section 1.7), we see that $b(n)\sim  n/h(\sqrt{ a  \mu n / (n K_2)})$. Since $m_n$ is regularly varying of index $3-\alpha 3/2$, $b$ is regularly varying of index $\beta>1/2$. 
The proof is now completed by observing that $c_a(n)\sim b(n)$. The asymptotic form of $c_a(n)$ and the dependence on $a$ now follows straightforwardly.
\end{proof}

\begin{proof}[Proof of Lemma \ref{lemma-sbc}]
%
 For large enough $n$, $b(n)c(n)\gg n$ so that we can write 
\[
\int_0^\infty f_n(x,b(n),c(n))dF(x) = \int_0^\infty \min \{xb(n)/(\mu n) , 1\} \min \{ xc(n)/(\mu n), 1\} dF(x).
\]
The RHS can be decomposed as follows: 
\[
 \frac {b(n)c(n)}{n^2 \mu^2} \int_0^{\mu n / c(n)}  x^2 d F(x)    + \frac{b(n)}{\mu n} \int_{\mu n / c(n)}^{\mu n / b(n)} x dF(x) + \bar F(\mu n / b(n)).
\]
Call these terms I,II, III. 
We can use Karamata's theorem to estimate Term I:
\[
 \frac {b(n)c(n)}{n^2 \mu^2} \int_0^{\mu n / c(n)}  x^2 d F(x) \sim L( \mu n / c(n)) \frac 1{2-\alpha} \frac {b(n)c(n)}{n^2 \mu^2} (\mu n / c(n))^{2-\alpha}.
\]
This is regularly varying with index $-[\alpha (1-\alpha_c) + \alpha_c-\alpha_b]$.
Term III is regularly varying of index $-\alpha (1-\alpha_b)$ and is therefore of small order of term $I$ if $\alpha_c>\alpha_b$, and behaves like Term I if $c(n)=b(n)$. 
Term II behaves like for some constant $K_{II}$, 
\[
\frac{b(n)}{\mu n} \int_{\mu n / c(n)}^\infty x dF(x) \sim K_{II}\frac{b(n)}{c(n)}\bar F( n/ c(n)),
\]
which is regularly varying of $-[\alpha (1-\alpha_c) + \alpha_c-\alpha_b]$, like Term I. Inspecting the slowly varying parts, it can be shown that
the asymptotic behavior of Term I and Term II is the same up to a constant. \end{proof}

\begin{proof}[Proof of Lemma \ref{lemma-sb}]
We write 
$$\int_0^\infty \int_x^\infty f_n(x,y,b(n))dF(y)dF(x)= (1/2)\int_0^\infty \int_0^\infty f_n(x,y,b(n))dF(x)dF(y)$$ and we
split the integral  $\int_0^\infty \int_0^\infty f_n(x,y,b(n))dF(x)dF(y)$ in two terms, where $xy<\mu n$ (Term I), and  $xy\geq \mu n$ (Term II).
\[
I= \frac 1{\mu n} \int_{xy< \mu n} xy \min \{xb(n)/(\mu n) , 1\} \min \{ yb(n)/(\mu n), 1\} dF(x)dF(y).
\]
We break up Term I into 3 more regions: a region where both $x,y$ are smaller than $\mu n/b(n)$ (Term Ii), a region where both $x,y$ are larger than $\mu n/b(n)$ (Term Iii), and 
the region where one is smaller, and one larger (Term Iiii).
For $n$ large enough, $\mu n/b(n)$ is much smaller than $\sqrt{\mu n}$, in which case $Ii$ equals
\begin{equation}
   \frac {b(n)^2}{(\mu n)^3} (\int_{x=0}^{ \mu n/b(n)}x^2 dF(x))^2 \sim  K \frac {b(n)^2}{(\mu n)^3} \left((\mu n/b(n))^2 \bar F(\mu n/b(n) )\right)^2= \frac {K\mu n}{b(n)^2}\bar F(\mu n/b(n) )^2.
\end{equation}
This term is regularly varying of index $-[(2\alpha_b-1)+2\alpha (1-\alpha_b)]$.

For Term Iii, we can lower and upper bound the region by respectively including constraints $x< \sqrt{ \mu n}$, $y< \sqrt{\mu n}$, and by removing the constraint 
$xy<\mu n$. In both cases, we end up with the square of an integral with the same asymptotic behavior, so we get
\begin{equation}
    Iii \sim \frac {1}{\mu n} (\int_{x=\mu n/b(n)}^\infty x d F(x))^2 \sim \frac {K}n (n/b(n))^2 \bar F (n/b(n))^2.
\end{equation}
This term is regularly varying of index $-[2(\alpha-1)(1-\alpha_b)+1]$, like Term Ii.
Term Iiii can be written as
\begin{equation}
    \frac {2b(n)}{(\mu n)^2} \int_{xy< \mu n, x < \mu n/b(n) <y} x^2y dF(x)dF(y).
\end{equation}
Again, the constraint  $xy< \mu n$ is asymptotically irrelevant, leading to the behavior
\begin{equation}
      \frac {2b(n)}{(\mu n)^2} \left(\int_0^{\mu n/b(n)} x^2 dF(x) \right) \left(\int_{\mu n/b(n)}^\infty ydF(y)\right) \sim \frac{Kb(n)}{n^2} (n/b(n))^3\bar F(n/b(n))^2.
\end{equation}
Once more, this term is regularly varying of index $-[2(\alpha-1)(1-\alpha_b)+1]$.
Note that $[2(\alpha-1)(1-\alpha_b)+1]<\alpha$ if $\alpha_b>1/2$.

We now turn to Term II. By bounding the two minima by 1, we see that
\[
II= \int_{xy> \mu n} \min \{xb(n)/(\mu n) , 1\} \min \{ yb(n)/(\mu n), 1\} dF(x)dF(y) \leq \Prob{XY>\mu n},
\]
with $X,Y$ iid and regularly varying of index $-\alpha$. Due to \cite{EmbrechtsGoldie}, the product is also regularly varying of index $-\alpha$. 
Since $\alpha > [2(\alpha-1)(1-\alpha_b)+1]$, Term II is asymptotically negligible. 
\end{proof}

\begin{proof}[Proof of Lemma \ref{lem:lowdegtriangleedges}]
For any given edge between vertices $i$ and $j$ of weights $w_i>w_j$ such that $w_iw_j<\mu n$, the number of triangles it is involved in is a binomial random variable with $n-2$ trials and probability at most
\begin{align}
    &\int_1^{\infty}\min\Big(\frac{w_iw_k}{\mu n},1\Big)\min\Big(\frac{w_jw_k}{\mu n},1\Big)dF(w_k)\nonumber\\
    & \leq \int_1^{\infty}h(w_i,w_j,w_k)w_k^{-\alpha}L(w_k)dw_k\nonumber\\
    & \leq \int_1^{\infty}h(w_i,w_j,w_k)w_k^{-\alpha+\delta}dw_k \nonumber\\
    & = \int_1^{\infty}\min\Big(\frac{w_iw_k}{\mu n},1\Big)\min\Big(\frac{w_jw_k}{\mu n},1\Big)w_k^{-\alpha-1+\delta}dw_k,
\end{align}
where $h(w_i,w_j,w_k)$ is a function such that $\int_1^{w_k}h(w_i,w_j,x)dx = \min(\frac{w_iw_k}{\mu n},1)\min(\frac{w_jw_k}{\mu n},1) $, and where in the second inequality we have used the Potter bound on the slowly varying function $L(x)$. Now
\begin{align}
    & \int_1^{\infty}w_k^{-\alpha-1+\delta}\min\Big(\frac{w_iw_k}{\mu n},1\Big)\min\Big(\frac{w_jw_k}{\mu n},1\Big)dw_k\nonumber\\
    & = \int_1^{\mu n/w_i}w_k^{1-\alpha+\delta}\frac{w_iw_j}{(\mu n)^2}dw_k + 
        \int_{\mu n/w_i}^{\mu n/ w_j}w_k^{-\alpha+\delta}\frac{w_j}{\mu n}dw_k + 
        \int_{\mu n/ w_j}^\infty w_k^{-\alpha-1+\delta}dw_k \nonumber\\
    & = \Bigg(\frac{\big(\frac{\mu n}{w_i}\big)^{2-\alpha+\delta}}{2-\alpha+\delta}\frac{w_iw_j}{(\mu n)^2} + \frac{\big(\frac{\mu n}{w_i}\big)^{1-\alpha+\delta}}{\alpha-\delta-1}\frac{w_j}{\mu n} + \frac{1}{\alpha-\delta}\Big(\frac{\mu n}{w_j}\Big)^{-\alpha+\delta}\Bigg)(1+o(1))\nonumber\\
    & = (\mu n)^{-\alpha+\delta}\Big(\frac{w_i^{\alpha-1}w_j}{(2-\alpha+\delta)(\alpha-\delta-1)}+ \frac{w_j^{\alpha-\delta}}{\alpha-\delta}\Big) (1+o(1)).
\end{align}
Now when $w_i\geq w_j$ and $w_iw_j\leq \mu n$, this expression is maximized by $w_i=w_j=\sqrt{\mu n}$ for $1<\alpha<2$. Thus, the number of triangles $\{i,j\}$ is involved in is a binomial random variable with probability at most 
\begin{align}
    C_1(\mu n)^{-\alpha+\delta}\sqrt{\mu n}^{-\alpha+\delta}(1+o(1)) \leq  C_2n^{(-\alpha+\delta)/2},
\end{align}
for some $C_1,C_2>0$ and $n$ sufficiently large, where we have used that $w_iw_j\leq \mu n$ and that $w_i>w_j$. Thus, the average number of triangles involving $\{i,j\}$ is at most $C_2n^{(2-\alpha+\delta)/2}$. When choosing $\delta$ sufficiently small such that $C_2n^{(2-\alpha)/2+\varepsilon}\gg n^{(2-\alpha+\delta)/2}$, Lemma~\ref{lem-concentration} yields that for $K>C_2n^{(2-\alpha)/2+\varepsilon}$,
\begin{equation}
	\Prob{\{i,j\} \text{ in }\geq  K  \text{ triangles}}\leq \exp(-c_1K)(1+o(1)),
\end{equation} 
for some $c_1>0$. 
\end{proof}


\section{Additional proofs for Section \ref{sec:prooftausmall}} 
\label{app:manyhubs}
\begin{proof}[Proof of Lemma~\ref{lem:b2edgetriang}]
	We first bound the number of $AB_2$ edges. As the connection probability is increasing in the vertex weights, the number of $AB_2$ edges and $A,B_2,B_2$ triangles is stochastically dominated by the number of such edges and triangles respectively when all $B_2$ vertices have the upper bound of their weights of $n/\log(n)^{\zeta/\alpha}$. We therefore assume that all $B_2$ vertices have weights $n/\log(n)^{\zeta/\alpha}$. 
 
 Now $\Prob{W>x} \sim C x^{-\alpha}$ implies that
  \begin{equation}\label{eq:tail3}
     \Prob{W>x}\leq  \tilde{C}x^{-\alpha}, \quad x\geq 1 
 \end{equation}
 for some $\hat{C}>0.$ By~\eqref{eq:tail3}, the average number of $AB_2$ edges can then be upper bounded by
	\begin{equation}
		N(B_2)n\hat{C}\Bigg(\int_{1}^{\log(n)^{\frac{\zeta}{\alpha}}}x^{-\alpha-1}\frac{x}{\log(n)^{\frac{\zeta}{\alpha}}}dx+\int_{\log(n)^{\frac{\zeta}{\alpha}}}^{\sqrt{n}}x^{-\alpha-1}dx\Bigg)\leq K_1 N(B_2)n\log(n)^{-\zeta},
	\end{equation}
for some $K_1>0$. 
	Now, let $X_{ij}$ denote the indicator that $i,j$ forms an edge and $i\in A, j\in B_2$. Define $X_{ij(uv)}=X_{ij}$ when $i\neq u$ and 0 otherwise. Because $d_u\leq M$,
	\begin{equation}
		\sum_{i\in A,j\in B_2}X_{ij}\leq M+\sum_{i\in A, j\in B_2}X_{ij(uv)}.
	\end{equation}
	Thus, by Lemma~\ref{lem:chatterjeelem}
	\begin{align}
		\Prob{E_{A,B_2} > n^{(\gamma+1)/2}} &=\Prob{\sum_{i\in A,j\in B_2}X_{ij} > n^{(\gamma+1)/2}}\nonumber\\
		& \leq \exp\Bigg(-\frac{n^{(\gamma+1)/2}}{M} \log\Big(\frac{n^{(\gamma-1)/2}}{N(B_2)\log(n)^{-\zeta}}\Big)\Bigg),
	\end{align}
which proves the first part of the lemma.
	
	Similarly, conditionally on the number of type $B_2$ vertices, the average number of $AB_2B_2$ triangles can be bounded by
	\begin{align}
		& N(B_2)^2n\hat{C}\Bigg(\int_{1}^{\log(n)^{\frac{\zeta}{\alpha}}}x^{-\alpha-1}\Big(\frac{x}{\log(n)^{\frac{\zeta}{\alpha}}}\Big)^2dx+\int_{\log(n)^{\frac{\zeta}{\alpha}}}^{\sqrt{n}}x^{-\alpha-1}dx\Bigg)\nonumber\\
  &\leq 
		 K_2 N(B_2)^2n\log(n)^{-\zeta},
	\end{align}
	for some $K_2>0$. 
	
	Now, for all $i\in A$ and $j,k\in B_2$, let $X_{ijk}$ denote the indicator that $i,j,k$ forms a triangle. Let $X_{ijk(uvw)}=X_{ijk}$ when $i\neq u$, and 0 otherwise. Then, $X_{ijk(uvw)}\leq X_{ijk}$, and as edges between $j$ and $k$ vertices have weights $n/\log(n)^{\zeta/\alpha}$ and therefore are present with probability one, $\sum_{ijk}X_{ijk(uvw)}$ is independent from $X_{uvw}$. Furthermore,
	\begin{equation}
		\sum_{i,j,k}X_{ijk}\leq d_u^2+\sum_{i,j,k}X_{ijk(uvw)}\leq M^2+\sum_{i,j,k}X_{ijk(uvw)}.
	\end{equation}
	Therefore, Lemma~\ref{lem:chatterjeelem} yields
	\begin{align}
		\Prob{ \triangle_{A,B_2,B_2}> n^\gamma}& \leq \exp\Big(-\frac{n^\gamma}{M^2}\log\Big(\frac{n^{\gamma-1}}{3K_1 N(B_2)^2\log(n)^{-\zeta}}))\Big).
	\end{align}
\end{proof}

\begin{proof}[Proof of Lemma~\ref{lem:Adegrees}]
	As the degree of a vertex of lower weights is stochastically dominated by the degree of a vertex with higher weight, we assume that vertex $i$ has the maximal weight $\sqrt{\mu n}$.
	Given the weights, the degree of a vertex with weight $\sqrt{\mu n}$ is a sum of independent Bernoulli random variables with mean $\sum_{i}\min(W_i/\sqrt{\mu n},1)$. We now compute the probability of the event that this mean is large. 
	First of all,
		\begin{align}
		\Var{W\ind{W\leq \sqrt{\mu n}}}\leq \Exp{W^2\ind{W\leq \sqrt{\mu n}}}\leq \hat{C}\int_1^{\sqrt{\mu n}}x^{1-\alpha}dx =\hat{C} (\mu n)^{(2-\alpha)/2}(1+o(1)).
	\end{align}
Thus, by Bernsteins' inequality,
	\begin{align}
		\Prob{|\sum_{i}W_i\ind{W_i\leq \sqrt{\mu n}}-\mu n|> \mu n} & \leq \exp\Big(-\frac{(\mu n)^2}{2n\hat{C}(\mu n)^{(2-\alpha)/2}+2(\mu n)^{3/2}}\Big)\nonumber\\
  & \leq \exp(-(\mu n)^{1/2})\nonumber\\
  & \leq \exp(-n^{(\gamma-1)/2}),
	\end{align}
	for all $\alpha\in(1,2)$ and $\gamma\leq 2$. 
For $\gamma>2$, and $n$ sufficiently large, Bernstein's inequality yields
	\begin{align}
		\Prob{|\sum_{i}W_i\ind{W_i\leq \sqrt{\mu n}}-\mu n|> n^{1/2+\gamma/3}}& \leq 
  \exp(-\frac{n^{1+2\gamma/3}}{(1+o(1))2(\mu n)^{3/2}})\nonumber\\
  & \leq \exp\Big(-n^{(\gamma-1)/2}\Big).
	\end{align} 
Thus, 
	\begin{equation}
	\Prob{\sum_{i}\frac{W_i}{\sqrt{\mu n}}\ind{W_i\leq \sqrt{\mu n}}> 2\sqrt{\mu n}+n^{\gamma/3}}\leq \exp\Big(-n^{(\gamma-1)/2}\Big).
\end{equation} 
Thus, the probability that the mean degree of a vertex with weight $W_i=\sqrt{\mu n}$ is at most $2\sqrt{\mu n}$ or $n^{\gamma/3}$ is at most $\exp(-n^{(\gamma-1)/2})$. On this event, the degree of a vertex of weight $\sqrt{\mu n}$ is a sum of indicators with mean at most $2\sqrt{\mu n}$ or $n^{\gamma/3}$.

	We use Lemma~\ref{lem-concentration} and the union bound to show that for $\gamma<2$
	\begin{align}
		\Prob{\bar{\mathcal{F}}}& \leq \sum_{i=1}^n\Prob{d_i>4\sqrt{\mu n}\mid W_i\leq\sqrt{\mu n}}\nonumber\\
		&\leq n\Prob{d_i>4\sqrt{\mu n}\mid W_i= \sqrt{\mu n}} \leq n\exp(-2\sqrt{\mu n}(2\log(2)-1)).
	\end{align}
		Similarly, for $\gamma>2$,
	\begin{align}
		\Prob{\bar{\mathcal{F}}}& \leq \sum_{i=1}^n\Prob{d_i>4n^{\gamma/3}\mid W_i =  \sqrt{\mu n}}
	\leq n\exp(-2n^{\gamma/3}(2\log(2)-1)).
	\end{align}
	Thus, as $\gamma/3>(\gamma-1)/2$ for $\gamma\in(1,3)$,
	\begin{equation}
		\Prob{\bar{\mathcal{F}}}\leq \exp\Big(-Dn^{(\gamma-1)/2} \log(n^{(\gamma-3)/2+\alpha})\Big).
	\end{equation}
\end{proof}


\section{Proof of Theorem~\ref{thm:triangldptausmall}: many dominating hubs}\label{sec:prooftausmallgamma}
\begin{proof}[Proof of Theorem~\ref{thm:triangldptausmall}]
 \textit{Lower bound.}
  As a lower bound, we compute the number of triangles with at least one vertex of weight $>\mu  n$. By Lemma~\ref{lem:nvertices}, the probability that at least $n^{\lambda}$ vertices of weights at least $\mu n$ are present can be bounded by 
\begin{equation}
    \Prob{n^\lambda \text{ vertices of weight }>\mu n}\geq \frac{1}{\sqrt{2n}}\exp(-n^\lambda\log(n^{\lambda+\alpha-1}))(1+o(1)).
\end{equation}
Now, $n^\lambda$ vertices of weight at least $\mu n$ generate $n^{1+2\lambda}/2$ triangles. Indeed, every pair of 2 hubs forms a triangle with any of the other $n$ vertices, so $n^{1+2\lambda}/2$ in total. Therefore, $\sqrt{2a}n^{(\gamma-1)/2}$ vertices of weight at least $n$ create $an^\gamma$ triangles. Thus, 
\begin{equation}\label{eq:trianglb}
    \Prob{\triangle_n>an^\gamma}\geq  \frac{1}{\sqrt{2 n}} \exp\big(-\sqrt{2a}n^{(\gamma-1)/2}\log\big(\frac{\gamma-3}{2}+\alpha\big)\big)(1+o(1)).
\end{equation}

\textit{Upper bound.}
We now distinguish and bound different types of triangles. 


\textit{BBB: all weights larger than $\sqrt{\mu n}$.}
Let $X_{u,v,w}$ denote the event that $u,v,w$ forms a triangle and that $w_u,w_v, w_w>\sqrt{\mu n}$. Then,
\begin{equation}
    \sum_{u,v,w}X_{u,v,w}\leq \sum_{u,v,w}\ind{w_u,w_v,w_w>\sqrt{\mu n}}=\Big(\sum_u\ind{w_u>\sqrt{\mu n}}\Big)^3.
\end{equation}
By Lemma~\ref{lem:nvertices}, when $\gamma>3-3\alpha/2,$
\begin{align}
\Prob{\sum_{u,v,w}X_{u,v,w}>\varepsilon n^\gamma}& \leq     \Prob{ \varepsilon^{1/3}n^{\gamma/3} \text{ vertices of weight }>\sqrt{\mu n} }\nonumber\\
    & \leq \exp\Big(-\varepsilon^{1/3}n^{\gamma/3} \log(n^{\gamma/3-(2-\alpha)/2})\Big),
\end{align}
by Lemma~\ref{lem:nvertices}. 
Now for fixed $\varepsilon>0$ and $\gamma>1$ therefore 
\begin{equation}\label{eq:triangbbb}
    \Prob{\triangle_{BBB}}\leq \exp\Big(-\sqrt{2a} n^{(\gamma-1)/2} \log(n^{(\gamma-3)/2+\alpha})\Big),
\end{equation}
for $n$ sufficiently large.

\textit{AAA: all weights smaller than $\sqrt{\mu n}$.}
By Lemma~\ref{lem:Adegrees}, we may work on the event $\mathcal{F}$, so that all degrees are bounded by $M=4\max(\sqrt{n},n^{\gamma/3})$. 
We aim to design two sets of indicators that deal with the dependencies between the presences of different triangles, so that we can use Lemma~\ref{lem:chatterjeelem}. Let $Y_{uvw}$ denote the indicator that $u,v,w$ forms a triangle and that $u,v,w$ have degree at most $\sqrt{\mu n}$. Now $Y_{uvw}\leq X_{uvw}$. We now define a set of indicators $Y_{xyz(uvw)}$. 

When $|\{x,y,z,u,v,w\}|=6$, we set $Y_{xyz(uvw)}=Y_{xyz}$, and otherwise, we set $Y_{xyz(uvw)}=0$. 
Now $\sum_{x,y,z}Y_{xyz(uvw)}$ is independent of $X_{uvw}$, as none of the entries of the summation depend on the edges $uv, uw, vw$. 

Furthermore, $\sum_{x,y,z}Y_{xyz(uvw)}\leq \sum_{x,y,z}Y_{x,y,z}$. Finally, 
\begin{equation}
    \sum_{x,y,z}Y_{x,y,z}\leq 3M^2+ \sum_{x,y,z}Y_{xyz(uvw)},
\end{equation}
as at most $M^2$ triangles involve vertex $u$ since its maximal degree is $M$. Thus, by Lemma~\ref{lem:chatterjeelem},
\begin{equation}
    \Prob{\sum_{u,v,w}Y_{uvw}>\varepsilon n^\gamma}\leq \exp(-\frac{\varepsilon n^\gamma}{M^2}\log(n^{\gamma-\frac{(2-\alpha)}{2}})).
\end{equation}
Again, for $\gamma>1$ and $n$ sufficiently large this indicates that
\begin{equation}\label{eq:triangaaa}
    \Prob{\triangle_{AAA}}\leq \exp\Big(-\sqrt{2a} n^{(\gamma-1)/2} \log(n^{(\gamma-3)/2+\alpha})\Big),
\end{equation}
as $\gamma-2\gamma/3>(\gamma-1)/2$ for $\gamma\in(1,3)$.


\textit{ABB triangles.}

To bound the number of these triangles, we split the $B$ vertices into $B_1$ vertices, $B_2'=B\cap B_2$ and $B_3'=B\cap B_3$.
We first investigate the number of $AB_3'B_3'$ triangles. When $B\cap B_3=\emptyset$, we are done. Otherwise, $A\subseteq B_3$ and by~\eqref{eq:triangb3b3b3}, for $\gamma>3-3\alpha/2$ and $D=\sqrt{2a}/h(1+\varepsilon)$,
\begin{equation}\label{eq:ab3b3}
	\Prob{\triangle_{A,B_3',B_3'}>\varepsilon n^\gamma}\leq \Prob{\triangle_{B_3,B_3,B_3}>\varepsilon n^\gamma}\leq \exp(-\sqrt{2a} n^{(\gamma-1)/2}\log(n^{(\gamma-3)/2+\alpha})).
\end{equation}

%

We now bound the number of $AB_1B_1$ type triangles by $n$ times the number of $B_1$ vertices squared. This is the number of ways to choose 2 type $B_1$ vertices, and one other vertex. Thus, by~\eqref{eq:Nb1},
\begin{align}\label{eq:triangabb}
    \Prob{\triangle_{AB_1B_1}>an^\gamma}& \leq \Prob{N(B_1)>\sqrt{2a}n^{(\gamma-1)/2}}\nonumber\\
    & \leq \exp\Big(-\sqrt{2a}n^{(\gamma-1)/2} \log(n^{(\gamma-3)/2+\alpha}\log(n)^{-\zeta})\Big).
\end{align}
Consequently,
\begin{equation}\label{eq:Ab1b1}
    \lim_{n\to\infty}\frac{\log\Big(\Prob{\triangle_{AB_1B_1}>a n^\gamma}\Big)}{n^{(\gamma-1)/2}\log(n)}=\sqrt{2a}\Big(\frac{\gamma-3}{2}+\alpha\Big).
\end{equation}

We bound the number of $AB_3'B_i'$ for $i=1,2$ by the number of $B_i'$ vertices times the number of edges between $AB_3'$ vertices. Again, when $B_3'$ is empty, we are done. Otherwise, $A\subseteq B_3$, so that by~\eqref{eq:edb3b3} 
\begin{equation}
    \Prob{E_{A,B_3'}> n^{(\gamma+1)/4}}\leq \Prob{E_{B_3,B_3}> n^{(\gamma+1)/4}}\leq \exp(-\sqrt{2a} n^{(\gamma-1)/2}\log(n^{(\gamma-3)/2+\alpha})).
\end{equation}
On this event, by~\eqref{eq:Nb1b2}, and for $n$ sufficiently large,
\begin{align}\label{eq:triangab2b3}
    \Prob{\triangle_{AB_3'B_i'}>\varepsilon n^\gamma}& \leq \Prob{N(B_1\cup B_2)>\varepsilon n^{(3\gamma-1)/4}}\nonumber\\
    & \leq \exp\Big(-\varepsilon n^{(3\gamma-1)/4} \log(n^{(\gamma+1)/4}\log(n^{(\gamma-3)/2+\alpha})^{-1})\Big)\nonumber\\
    & \leq \exp\Big(-\sqrt{2a} n^{(\gamma-1)/2} \log(n^{(\gamma-3)/2+\alpha}\log(n)^{-\alpha})\Big).
\end{align}

We now bound the number of $AB_2'B_1$ triangles by $N(B_1)$ times the number of $AB_2$ edges. 
By Lemma~\ref{lem:b2edgetriang} with $M=\max(\sqrt{n},n^{\gamma/3})$, and for $n$ sufficiently large
\begin{align}\label{eq:eab2p}
    \Prob{E_{A,B_2'} > \frac{\varepsilon}{\sqrt{2a}} n^{(\gamma+1)/2}} & \leq \Prob{E_{A,B_2} > \frac{\varepsilon}{\sqrt{2a}}n^{(\gamma+1)/2}}\nonumber\\
    & \leq \exp(-\sqrt{2a}n^{(\gamma-1)/2}\log(n^{(\gamma-3)/2+\alpha}))
\end{align}
as long as $N(B_2)\leq n^{(\gamma-1)/2}\log(n)^{\zeta}$. By~\eqref{eq:Nb1b2} this happens with probability 
\begin{align}\label{eq:NB1}
	\Prob{N(B_2)>n^{(\gamma-1)/2}\log(n)^{\zeta}}
	& \leq \exp\Big(-n^{(\gamma-1)/2}\log(n)^{\zeta}\log\Big(\log(n)^{\zeta}\log(n^{(\gamma-3)/2+\alpha})^{-1}\Big)\Big)\nonumber\\
	& \leq \exp\Big(-\sqrt{2a}n^{(\gamma-1)/2} \log(n^{(\gamma-3)/2+\alpha})\Big),
\end{align}
 when $\zeta>1$ and $n$ is sufficiently large. 
Thus, on the event of~\eqref{eq:eab2p}, by~\eqref{eq:Nb1},
\begin{align}\label{eq:triangab2b1}
    \Prob{\triangle_{AB_1B_2'}>\varepsilon n^\gamma}& \leq \Prob{N(B_1)E_{AB_2'}>\varepsilon n^{\gamma}}\nonumber\\
    & \leq \Prob{N(B_1)>\sqrt{2a} n^{(\gamma-1)/2}}\nonumber\\
    & \leq \exp\Big(-\sqrt{2a} n^{(\gamma-1)/2} \log(n^{(\gamma-3)/2+\alpha}\log(n)^{-\zeta})\Big).
\end{align}
We now consider A$B_2'B_2'$ triangles. 
By Lemma~\ref{lem:b2edgetriang},
\begin{align}\label{eq:ab2b2triang}
	\Prob{ \triangle_{A,B_2',B_2'}> n^\gamma}& \leq \exp\Big(-\frac{n^\gamma}{M^2}\log(n^\gamma/(3K_1 N(B_2')^2n\log(n)^{-\zeta}))\Big)\nonumber\\
	& \leq \exp\Big(-\sqrt{2a}n^{(\gamma-1)/2} \log(n^{(\gamma-3)/2+\alpha})\Big),
\end{align}
as long as $N(B_2)\leq n^{(\gamma-1)/2}\log(n)^{\zeta}$ and $\gamma>1$, which happens with probability~\eqref{eq:NB1}.

Thus, ~\eqref{eq:Ab1b1},~\eqref{eq:ab3b3}, ~\eqref{eq:triangab2b3},~\eqref{eq:triangab2b1} and~\eqref{eq:ab2b2triang} yield that 
\begin{equation}\label{eq:triangabbfinal}
    \lim_{n\to\infty}\frac{\log\Big(\Prob{\triangle_{ABB}>n^\gamma}\Big)}{n^{(\gamma-1)/2}\log(n)}=\sqrt{2a}\Big(\frac{\gamma-3}{2}+\alpha\Big).
\end{equation}

\textit{AAB triangles.} When bounding the number of $AAB$ triangles, 
we work on the event $\mathcal{F}$ from Lemma~\ref{lem:Adegrees}, so that all type-A vertices have degrees at most $M=\max(4\sqrt{\mu n},4n^{\gamma/3})$. Again, we split the B vertices into $B_1', B_2'$ and $B_3'$. 

We now bound the number of $AAB_i'$ triangles with $i\in B_1'\cap B_2'$. To do so, we first show that the event $\mathcal{E}$ that there are at most $ n^{(\gamma+1)/2}$ edges between type A vertices happens with high probability. Let $X'_{ij}$ denote the indicator that an edge is present between $i$ and $j$, let $X_{ij}$ be the indicator that $i,j\in A$ and that $\{i,j\}$ is an edge. Finally, we define $X_{ij(uv)}=X_{ij}$ when $i,j,u,v$ are all distinct, and set $X_{ij(uv)}=0$ when $i$ or $j$ overlaps with $u$ or $v$. Then $\sum_{ij}X_{ij(uv)}$ is independent from $X_{uv}$, and furthermore
\begin{equation}
   E_{A_,A}:=\sum_{ij}X_{ij} \leq d_u+d_v+ \sum_{ij}X_{ij(uv)}\leq 2M +\sum_{ij}X_{ij(uv)}
\end{equation}
Then, Lemma~\ref{lem:chatterjeelem} shows that
\begin{align}
    \Prob{E_{A,A}>n^{(3\gamma-1)/4}} & \leq \exp\Big(-\frac{n^{(3\gamma-1)/4}}{2n^{\gamma/3}}\log\Big(\frac{n^{(3\gamma-1)/2}}{3\mu n}\Big)\Big)\nonumber\\
    & < \exp\Big(-\sqrt{2a} n^{(\gamma-1)/2} \log(n^{(\gamma-3)/2+\alpha})\Big),
\end{align}
for $\gamma\in[2,3)$. Furthermore, 
\begin{align}\label{eq:eaagammasmall}
    \Prob{\mathcal{E}}: = \Prob{E_{A,A}>4\mu n} & \leq \exp\Big(-\frac{4\mu n}{8K\sqrt{\mu n}}\log\Big(\frac{4\mu n}{3\mu n}\Big)\Big)\nonumber\\
    & < \exp\Big(-\sqrt{2a}n^{(\gamma-1)/2} \log(n^{(\gamma-3)/2+\alpha})\Big),
\end{align}
for $\gamma\in [1,2)$ and $n$ sufficiently large. 

We now condition on the event $\mathcal{E}$. On this event, we can bound the number of $AAB_i$ triangles for $i=1,2$ by $4\mu n$ times the number of type $B_i$ vertices for $\gamma\in[1,2)$. Thus, on $\mathcal{E}$, by Lemma~\ref{lem:nvertices}
\begin{align}\label{eq:triangaab1}
    \Prob{\triangle_{AAB_i}>\varepsilon n^\gamma} &\leq \Prob{N(B_2'\cup B_1)>\varepsilon n^{\gamma-1}/(4\mu)} \leq \exp\Big(-\frac{\varepsilon n^{\gamma-1}}{4\mu}\log(n^{\gamma-1-(\gamma-1)/2})\Big)\nonumber\\
    & \leq \exp\Big(-\sqrt{2a}n^{(\gamma-1)/2} \log(n^{(\gamma-1)/2-(2-\alpha)/2})\Big),
\end{align}
for $\gamma\in[1,2)$ and $n$ sufficiently large. Similarly, for $\gamma\in[2,3),$ by Lemma~\ref{lem:nvertices}
\begin{align}\label{eq:triangaab2}
    \Prob{\triangle_{AAB_i}>\varepsilon n^\gamma} &\leq \Prob{N(B_2'\cup B_1)>\varepsilon n^{(\gamma+1)/4}} \nonumber\\
    &\leq \exp(-\varepsilon n^{(\gamma+1)/4}\log(n^{(\gamma+1)/4-(\gamma-1)/2}))\nonumber\\
    & \leq \exp\Big(-\sqrt{2a}n^{(\gamma-1)/2} \log(n^{(\gamma-1)/2-(2-\alpha)/2})\Big),
\end{align}
for $n$ sufficiently large. 

Finally, we bound the number of $AAB_3'$ triangles. When $B_3'$ is empty, we are done. Otherwise, $A\subseteq B_3$, and any $AAB_3'$ triangle is also an $B_3B_3B_3$ triangle, whose number can be bounded by~\eqref{eq:ab3b3}.

To conclude,~\eqref{eq:triangbbb},~\eqref{eq:triangbbb},~\eqref{eq:triangabbfinal},~\eqref{eq:triangaab1} and~\eqref{eq:triangaab2} with a limit of $\varepsilon\downarrow 0$ yield that
\begin{equation}
    \log\big(\Prob{\triangle_n>a n^\gamma}\big)\leq -\sqrt{2a}n^{(\gamma-1)/2} \log(n^{(\gamma-3)/2+\alpha})(1+o(1)).
\end{equation}

Combining this with~\eqref{eq:trianglb} gives

\begin{equation}
    \lim_{n\to\infty}\frac {\log  \Prob{\triangle_n>n^\gamma}}{n^{(\gamma-1)/2}\log(n)} =\sqrt{2a}\Big(\frac{\gamma-3}{2}+\alpha\Big).
\end{equation}

\end{proof}

\section{Proof of Theorem \ref{thm:thetasmall}} 
\label{app:proof:thm:thetasmall}

The proof follows the same steps as the proof of Theorem \ref{thm-triangle-singlebigjump}. To avoid unnecessary repetitions, we restrict ourselves to
indicating which steps require nontrivial modifications. 
First of all, the asymptotic expansion $c_{n^\theta}(n) \sim L^*(n) n^{\beta+\frac\theta 2 \frac{1}{\alpha-1}}=o(n)$ follows from a straightforward modification of the proof of Lemma \ref{lemma-can}, and the $o(n)$ behavior follows from the inequality $\theta< \frac 32 \alpha-2$. 

The first major step is to prove that the analogue of Theorem \ref{thm-gn} holds, with $a$ replaced by $n^\gamma$. 
To establish this, we first show how to modify the proof of Proposition \ref{prop-onebigjumpneeded}, in particular how to manage the ten-term bound (\ref{eq-tenterms}). The three terms that require modification are the third, fifth and sixth term. For the third term, we can again invoke Lemma \ref{lemma-sb} to conclude that this term behaves like $c \varepsilon^{(\alpha-1)/2} n^\theta \mu_n$. 
For the fifth term, apply Lemma \ref{lemma-sbc} with $\alpha_b=\alpha_c= \beta + \frac {\theta}{2} \frac 1{\alpha-1}$, to conclude that 
this term is regularly varying with index $1-\alpha + \alpha^2/(4(\alpha-1))+ \alpha \theta/(2(\alpha-1))$. 
To show that the fifth term is of small order in $n$, it suffices to show that 
\[
1-\alpha + \alpha^2/(4(\alpha-1))+ \alpha \theta/(2(\alpha-1)) < 3-\alpha 3/2+\gamma,
\]
which is equivalent to 
\[
3\alpha^2 - 2\alpha + \theta (4-2\alpha) < 8(\alpha-1). 
\]
This follows from the inequality $\theta< \alpha 3/2-2$.
As before, the sixth term is of smaller order than the fifth term, and therefore is also negligible as $n$ grows large. 
This leads to the conclusion that
\begin{equation}
    \Prob{G_n > (1+n^\theta )m_n ; L_n(\varepsilon c_{n^\theta }(n))=0} = o( n \Prob{W_1> c_{n^\theta}(n)}),
\end{equation}
which is the desired extension of Proposition \ref{prop-onebigjumpneeded}.
The statement and the proof of Proposition \ref{prop-wlln} extends straightforwardly to handle the case where $a$ is replaced by $n^\theta$. 
This readily leads to the conclusion that 
\begin{equation}
        \Prob{G_n > m_n (1+n^\theta) } = (1+o(1)) n \Prob{ W > c_{n^\theta}(n)}. 
\end{equation}
To derive the same estimate for $\triangle_n$ we use the same steps as in Section \ref{sec:prooftaularge}: the proof of the asymptotic upper bound again follows from Lemma \ref{lem:expconditionedtausmallnew}. To derive an asymptotic lower bound, we need to modify the first part Lemma \ref{lem-lb-additionaltriangles}.
In particular, we need that 
\begin{equation}
 \triangle_n (\delta,n^\theta)/(m_n(1+n^\theta)) \rightarrow 1
\end{equation}
in probability. 
The proof of this statement follows by simply following the same steps as the first part of Lemma \ref{lem-lb-additionaltriangles}, with $a$ replaced
by $n^\theta$. With the appropriate analogues of Lemma's \ref{lem:expconditionedtausmallnew} and \ref{lem-lb-additionaltriangles} in place, the 
proof of Theorem \ref{thm:thetasmall} follows straightforwardly. 
\end{document}